\documentclass[final,3p]{elsarticle}
\usepackage{lineno,hyperref}
\usepackage{amsmath}
\usepackage{amssymb}
\usepackage[linesnumbered,ruled,vlined]{algorithm2e}
\usepackage{caption}
\usepackage{mathtools}
\usepackage{changes}
\usepackage{multirow}
\usepackage{verbatim}
\usepackage{mathrsfs}
\usepackage{graphicx}
\usepackage{subcaption}
\usepackage{amsmath,amsthm,bm,mathrsfs}
\theoremstyle{plain}
\newtheorem{theorem}{Theorem}[section]

\theoremstyle{definition}

\theoremstyle{remark}

\modulolinenumbers[5]

\journal{ArXiv.org}

\begin{document}

\begin{frontmatter}

\title{A Low-rank ADI Algorithm for the Numerical Solution of Large Discrete-time Non-symmetric Algebraic Riccati Equations}

\author[uz]{Umair~Zulfiqar\corref{mycorrespondingauthor}}
\cortext[mycorrespondingauthor]{Corresponding author}
\ead{umair@yangtzeu.edu.cn}
\address[uz]{School of Electronic Information and Electrical Engineering, Yangtze University, Jingzhou, Hubei, 434023, China}

\begin{abstract}
Discrete-time non-symmetric algebraic Riccati equations (DTNAREs) arise in game-theoretic computations of Nash equilibria. Solving such equations at large scale is often computationally prohibitive. This paper develops a numerical approach for large-scale DTNAREs whose solutions are low rank. A low-rank alternating direction implicit (ADI) method is introduced that recursively constructs a low-rank stabilizing solution without explicitly solving any projected DTNARE. Through the pole-placement property of the low-rank ADI iteration, the method ensures that the implicitly solved projected DTNARE always admits a stabilizing solution. An automatic shift-generation strategy is also developed for the ADI iterations. Once an initial shift is provided, the algorithm computes the low-rank solution without further user intervention. Numerical experiments on DTNAREs with dimensions between \(10^6\) and \(10^7\) demonstrate the accuracy and efficiency of the method. The results confirm that the proposed low-rank ADI algorithm is an effective solver for large-scale DTNAREs that would otherwise be computationally prohibitive.
\end{abstract}


\end{frontmatter}

\section{Introduction}
Game theory provides mathematical models for analyzing competitive situations in which the payoff of each participant depends on the decisions of all players. A central objective is the determination of equilibria, that is, strategy profiles that are mutually optimal for all players. Among such equilibria, the Nash equilibrium is the most widely used concept. The equilibrium that is obtained depends strongly on the information available to the players when they construct their strategies. In an open-loop information structure, players do not observe the current state and must commit in advance to a predetermined strategy, in contrast to a closed-loop structure. For infinite-horizon linear-quadratic games, computing a Nash equilibrium requires solving coupled algebraic Riccati equations, which can be reformulated as a single non-symmetric algebraic Riccati equation (NARE) \cite{abou2012matrix,dockner2000differential,ionescu1996reverse,juang1995existence}. The discrete-time non-symmetric algebraic Riccati equation (DTNARE) associated with these Nash strategies is studied in \cite{jungers2010general}. The algorithms proposed in \cite{jungers2010general} are, however, limited to small-dimensional problems. This paper is concerned with computing stabilizing solutions of large-scale DTNAREs that admit numerically low-rank solutions.

Consider the DTNARE
\begin{align}
AX\hat{A}-EX\hat{E}-AX\hat{B}\big(R+CX\hat{B}\big)^{-1}CX\hat{A}+B\hat{C}=0,\label{dtnare}
\end{align}
where \(E\in\mathbb{R}^{n\times n}\), \(A\in\mathbb{R}^{n\times n}\), \(B\in\mathbb{R}^{n\times m}\), \(C\in\mathbb{R}^{p\times n}\), \(\hat{E}\in\mathbb{R}^{\hat{n}\times \hat{n}}\), \(\hat{A}\in\mathbb{R}^{\hat{n}\times \hat{n}}\), \(\hat{B}\in\mathbb{R}^{\hat{n}\times p}\), \(\hat{C}\in\mathbb{R}^{m\times \hat{n}}\), \(R\in\mathbb{R}^{p\times p}\), and \(X\in\mathbb{R}^{n\times \hat{n}}\). The matrices \(E\), \(\hat{E}\), \(R\), and \(R+CX\hat{B}\) are assumed to be invertible. Under the assumption \(m,p\ll n,\hat{n}\), the solution of DTNARE \eqref{dtnare} is numerically low-rank; this is the main regime considered in the paper.

Introduce the gain matrices \(K\) and \(\hat{K}\) by
\begin{align}
K=AX\hat{B}\big(R+CX\hat{B}\big)^{-1},\qquad 
\hat{K}=\big(R+CX\hat{B}\big)^{-1}CX\hat{A}.
\end{align}
The corresponding closed-loop matrices \(A_{cl}\) and \(\hat{A}_{cl}\) are defined by
\begin{align}
A_{cl}=A-KC,\qquad \hat{A}_{cl}=\hat{A}-\hat{B}\hat{K}.
\end{align}
A matrix \(X_{\star}\) is called a stabilizing solution of DTNARE \eqref{dtnare} if the matrices \(E^{-1}A_{cl}\) and \(\hat{A}_{cl}\hat{E}^{-1}\) are Schur stable, that is, if all their eigenvalues lie inside the unit circle.
\section{Main Work}
Low-rank ADI methods are iterative techniques for solving large-scale matrix equations whose solutions can be represented compactly in low-rank form \cite{wachspress1988iterative,simoncini2016computational,benner2013numerical}. These methods have been extended to a broad class of linear and nonlinear matrix problems, including Lyapunov equations \cite{benner2013efficient,benner2013reformulated,zulfiqar2026new}, Sylvester equations \cite{benner2009adi,benner2014computing}, Stein equations \cite{benner2011numerical,zulfiqar2026}, continuous-time algebraic Riccati equations (CAREs) \cite{benner2018radi,bertram2024family,zulfiqar2025unified,zulfiqar2026ldl}, continuous-time nonsymmetric algebraic Riccati equations (CTNAREs) \cite{zulfiqar2026low}, and discrete-time algebraic Riccati equations (DAREs) \cite{zulfiqar2026d}. As far as is known, no low-rank ADI algorithm has yet been developed for DTNAREs of the form \eqref{dtnare}. The method proposed here approximates the solution of \eqref{dtnare} in the low-rank form
\[
X\approx V^{(k)} \tilde{X}^{(k)} (\hat{W}^{(k)})^\top,
\]
where \(V^{(k)}\in\mathbb{R}^{n\times km}\), \(\tilde{X}^{(k)}\in\mathbb{R}^{km\times km}\), \(\hat{W}^{(k)}\in\mathbb{R}^{\hat{n}\times km}\), and $k\ll n,\hat{n}$.

As demonstrated in earlier works \cite{zulfiqar2026new,zulfiqar2026,zulfiqar2025unified,zulfiqar2026ldl,zulfiqar2026low,zulfiqar2026d}, low-rank ADI algorithms—despite their different origins and developments—can be interpreted as recursive rational interpolation methods based on Petrov–Galerkin projection. The interpolation points in these methods are the mirror images of the ADI shifts. The main distinction among existing low-rank ADI solvers lies in their pole-placement behavior, as discussed below.

Let \(\{\alpha_i\}_{i=1}^{k}\) and \(\{\beta_i\}_{i=1}^{k}\) denote the ADI shift parameters used in constructing the low-rank approximation \(X\approx V^{(k)} \tilde{X}^{(k)} (\hat{W}^{(k)})^\top\). In low-rank ADI methods, the trial bases \(V^{(k)}\) and \(\hat{W}^{(k)}\) satisfy
\begin{align}
\operatorname{span}_{i=1,\dots,k}\left\{(-\alpha_i E - A)^{-1} B\right\} &\subset \mathrm{Ran}(V^{(k)}),\label{span_prop1}\\
\operatorname{span}_{i=1,\dots,k}\left\{(-\beta_i \hat{E}^\top - \hat{A}^\top)^{-1} \hat{C}^\top\right\} &\subset \mathrm{Ran}(\hat{W}^{(k)}).\label{span_prop2}
\end{align}
The residual corresponding to the approximation \(X\approx \bar{X}^{(k)}= V^{(k)} \tilde{X}^{(k)} (\hat{W}^{(k)})^\top\) is
\begin{align}
R_x^{(k)} = A\bar{X}^{(k)}\hat{A}-E\bar{X}^{(k)}\hat{E}-A\bar{X}^{(k)}\hat{B}\big(R+C\bar{X}^{(k)}\hat{B}\big)^{-1}C\bar{X}^{(k)}\hat{A}+B\hat{C}.\label{residual}
\end{align}
For suitable test bases \(W^{(k)}\) and \(\hat{V}^{(k)}\), which are not explicitly formed and satisfy
\[
(W^{(k)})^\top E V^{(k)} = I\quad \text{and}\quad (\hat{W}^{(k)})^\top \hat{E} \hat{V}^{(k)} = I,
\]
the residual produced by a general low-rank ADI method satisfies the Petrov–Galerkin condition
\[
(W^{(k)})^\top R_x^{(k)} \hat{V}^{(k)} = 0.
\]

Equivalently, \(\tilde{X}^{(k)}\) satisfies the projected DTNARE
\begin{align}
A_r^{(k)}\tilde{X}^{(k)}\hat{A}_r^{(k)}-\tilde{X}^{(k)}-A_r^{(k)}\tilde{X}^{(k)}\hat{B}_r^{(k)}\big(R+C_r^{(k)}\tilde{X}^{(k)}\hat{B}_r^{(k)}\big)^{-1}C_r^{(k)}\tilde{X}^{(k)}\hat{A}_r^{(k)}+B_r^{(k)}\hat{C}_r^{(k)}=0,\label{proj_dtnare}
\end{align}
where the projected matrices are
\begin{align}
(W^{(k)})^\top E V^{(k)}&=I,\quad A_r^{(k)}=(W^{(k)})^\top A V^{(k)}, \quad B_r^{(k)}=(W^{(k)})^\top B,\quad C_r^{(k)}=C V^{(k)},\nonumber\\
(\hat{W}^{(k)})^\top \hat{E} \hat{V}^{(k)}&=I,\quad \hat{A}_r^{(k)}=(\hat{W}^{(k)})^\top \hat{A} \hat{V}^{(k)}, \quad \hat{B}_r^{(k)}=(\hat{W}^{(k)})^\top \hat{B},\quad \hat{C}_r^{(k)}=\hat{C} \hat{V}^{(k)}.
\end{align}
Define the full-order and reduced-order transfer functions
\begin{align}
G(z)&=C(zE-A)^{-1}B,& \hat{G}(z)&=\hat{C}(z\hat{E}-\hat{A})^{-1}\hat{B},\nonumber\\
G_r^{(k)}(z)&=C_r^{(k)}(zI-A_r^{(k)})^{-1}B_r^{(k)},& \hat{G}_r^{(k)}(z)&=\hat{C}_r^{(k)}(zI-\hat{A}_r^{(k)})^{-1}\hat{B}_r^{(k)}.\nonumber
\end{align}

The subspace conditions \eqref{span_prop1} and \eqref{span_prop2} imply the interpolation conditions
\begin{align}
G(-\alpha_i)=G_r^{(k)}(-\alpha_i)\quad \text{and} \quad  \hat{G}(-\overline{\beta}_i)=\hat{G}_r^{(k)}(-\overline{\beta}_i),\label{int_cond}
\end{align}
for \(i=1,\dots,k\) \cite{beattie2017chapter}.

The projected gain matrices are
\begin{align}
K_r^{(k)}&=A_r^{(k)}\tilde{X}^{(k)}\hat{B}_r^{(k)}\big(R+C_r^{(k)}\tilde{X}^{(k)}\hat{B}_r^{(k)}\big)^{-1},\quad 
\hat{K}_r^{(k)}=\big(R+C_r^{(k)}\tilde{X}^{(k)}\hat{B}_r^{(k)}\big)^{-1}C_r^{(k)}\tilde{X}^{(k)}\hat{A}_r^{(k)}.
\end{align}
The corresponding projected closed-loop matrices are
\begin{align}
A_{r,cl}^{(k)}=A_r^{(k)}-K_r^{(k)}C_r^{(k)},\quad \hat{A}_{r,cl}^{(k)}=\hat{A}_r^{(k)}-\hat{B}_r^{(k)}\hat{K}_r^{(k)}.
\end{align}
For the proposed low-rank ADI solver for DTNARE \eqref{dtnare}, the desired pole-placement property is that the eigenvalues of \(A_{r,cl}^{(k)}\) and \(\hat{A}_{r,cl}^{(k)}\) lie inside the unit circle. This ensures that \(\tilde{X}^{(k)}\) is the stabilizing solution of the projected DTNARE \eqref{proj_dtnare}.

As noted in \cite{zulfiqar2025unified}, another important property of existing low-rank ADI methods is that the solutions of the projected matrix equations have a block-diagonal structure, with each block associated with an ADI shift pair $(\alpha_i,\beta_i)$. Since these blocks can be computed via analytical expressions, the projected matrix equations are solved implicitly rather than directly. Existing low-rank ADI algorithms construct the low-rank solution recursively by adding new terms while keeping previously computed terms unchanged. The main computational work consists of shifted linear solves and elementary matrix operations. The following subsection develops a low-rank ADI solver for the DTNARE \eqref{dtnare} that retains these structural and computational properties.
\subsection{A Low-rank ADI Solver}
Suppose that the shifts \(\alpha_i\) and \(\beta_i\) are grouped into the following four cases, following the same grouping used in the ADI method for Sylvester equations \cite{benner2014computing}:
\begin{enumerate}
\item Case I: \(\mathrm{Im}(\alpha_i)=0\) and \(\mathrm{Im}(\beta_i)=0\).
\item Case II: \(\mathrm{Im}(\alpha_i)\neq 0\), \(\alpha_{i+1}=\overline{\alpha}_i\), \(\mathrm{Im}(\beta_i)\neq0\), and \(\beta_{i+1}=\overline{\beta}_i\).
\item Case III: \(\mathrm{Im}(\alpha_i)\neq 0\), \(\alpha_{i+1}=\overline{\alpha}_i\), \(\mathrm{Im}(\beta_i)=0\), and \(\mathrm{Im}(\beta_{i+1})=0\).
\item Case IV: \(\mathrm{Im}(\alpha_i)=0\), \(\mathrm{Im}(\alpha_{i+1})=0\), \(\mathrm{Im}(\beta_i)\neq0\), and \(\beta_{i+1}=\overline{\beta}_i\).
\end{enumerate}

Based on this grouping, define the matrices \(s_v^{(i)}\), \(s_w^{(i)}\), and \(l^{(i)}\) as follows:
\begin{enumerate}
\item Case I: 
\begin{align}
s_v^{(i)} = -\alpha_i I_m,\quad s_w^{(i)} = -\beta_i I_m,\quad l^{(i)} = -I_m.\label{sl1}
\end{align}
\item Case II: 
\begin{align}
s_v^{(i)} = \begin{bmatrix}
-\mathrm{Re}(\alpha_i)I_m & -\mathrm{Im}(\alpha_i)I_m \\
\mathrm{Im}(\alpha_i)I_m & -\mathrm{Re}(\alpha_i)I_m
\end{bmatrix}\quad
s_w^{(i)} = \begin{bmatrix}
-\mathrm{Re}(\beta_i)I_m & -\mathrm{Im}(\beta_i)I_m \\
\mathrm{Im}(\beta_i)I_m & -\mathrm{Re}(\beta_i)I_m
\end{bmatrix},\quad
l^{(i)} = \begin{bmatrix} -I_m & 0 \end{bmatrix}.\label{sl2}
\end{align}
\item Case III: 
\begin{align}
s_v^{(i)} = \begin{bmatrix}
-\mathrm{Re}(\alpha_i)I_m & -\mathrm{Im}(\alpha_i)I_m \\
\mathrm{Im}(\alpha_i)I_m & -\mathrm{Re}(\alpha_i)I_m
\end{bmatrix},\quad
s_w^{(i)} = \begin{bmatrix}
-\beta_i I_m & I_m \\
0 & -\beta_{i+1} I_m
\end{bmatrix},\quad
l^{(i)} &= \begin{bmatrix} -I_m & 0 \end{bmatrix}.\label{sl3}
\end{align}
\item Case IV: 
\begin{align}
s_v^{(i)} &= \begin{bmatrix}
-\alpha_i I_m & I_m \\
0 & -\alpha_{i+1} I_m
\end{bmatrix},\quad
s_w^{(i)} = \begin{bmatrix}
-\mathrm{Re}(\beta_i)I_m & -\mathrm{Im}(\beta_i)I_m \\
\mathrm{Im}(\beta_i)I_m & -\mathrm{Re}(\beta_i)I_m
\end{bmatrix},\quad
l^{(i)} = \begin{bmatrix} -I_m & 0 \end{bmatrix}.\label{sl4}
\end{align}
\end{enumerate}
Define \(S_v^{(i)}\), \((S_v^{(i)})^{-1}\), \(L^{(i)}\), \(S_w^{(i)}\), \((S_w^{(i)})^{-1}\), \(\tilde{X}^{(i)}\), \(V^{(i)}\), \(\hat{W}^{(i)}\), \(B_{\perp}^{(i)}\), and \(\hat{C}_{\perp}^{(i)}\) by
\begin{align}
S_v^{(i)}&=\begin{bmatrix}S_v^{(i-1)}&S_{v,12}^{(i)}\\\mathbf{0}&s_v^{(i)}\end{bmatrix},\quad (S_v^{(i)})^{-1}=\begin{bmatrix}(S_v^{(i-1)})^{-1}&-(S_v^{(i-1)})^{-1}S_{v,12}^{(i)}(s_v^{(i)})^{-1}\\\mathbf{0}&(s_v^{(i)})^{-1}\end{bmatrix},\quad L^{(i)}=\begin{bmatrix}L^{(i-1)}&l^{(i)}\end{bmatrix},\nonumber\\
S_w^{(i)}&=\begin{bmatrix}S_w^{(i-1)}&S_{w,12}^{(i)}\\\mathbf{0}&s_w^{(i)}\end{bmatrix},\quad (S_w^{(i)})^{-1}=\begin{bmatrix}(S_w^{(i-1)})^{-1}&-(S_w^{(i-1)})^{-1}S_{w,12}^{(i)}(s_w^{(i)})^{-1}\\\mathbf{0}&(s_w^{(i)})^{-1}\end{bmatrix},\quad \tilde{X}^{(i)}=\begin{bmatrix}\tilde{X}^{(i-1)}&\mathbf{0}\\\mathbf{0}&\tilde{x}_i\end{bmatrix},\nonumber\\
V^{(i)}&=\begin{bmatrix}V^{(i-1)}&v_i\end{bmatrix},\quad \hat{W}^{(i)}=\begin{bmatrix}\hat{W}^{(i-1)}&\hat{w}_i\end{bmatrix},\quad B_{\perp}^{(i)}=B-EV^{(i)}\tilde{X}^{(i)}(S_w^{(i)})^{-\top}(L^{(i)})^\top,\nonumber\\ \hat{C}_{\perp}^{(i)}&=\hat{C}-L^{(i)}(S_v^{(i)})^{-1}\tilde{X}^{(i)}(\hat{W}^{(i)})^\top \hat{E},\quad K_v^{(i)}=EV^{(i)}\tilde{X}^{(i)}(S_w^{(i)})^{-\top}(\hat{W}^{(i)})^\top\hat{B}R^{-1},\nonumber\\
\hat{K}_w^{(i)}&=R^{-1}CV^{(i)}(S_v^{(i)})^{-1}\tilde{X}^{(i)}(\hat{W}^{(i)})^\top\hat{E},\label{big_S_L}
\end{align}
where
\begin{align}
S_{v,12}^{(i)}=\tilde{X}^{(i-1)}(S_w^{(i-1)})^{-\top}\Big((L^{(i-1)})^\top l^{(i)}+(\hat{W}^{(i-1)})^\top \hat{B}R^{-1}C v_i\Big),\label{Sv12}\\
S_{w,12}^{(i)}=(\tilde{X}^{(i-1)})^\top(S_v^{(i-1)})^{-\top}\Big((L^{(i-1)})^\top l^{(i)}+(V^{(i-1)})^\top C^\top R^{-\top}\hat{B}^\top \hat{w}_i\Big),\label{Sw12}\\
(s_w^{(i)})^\top (\tilde{x}_i)^{-1} s_v^{(i)}- (\tilde{x}_i)^{-1}-(l^{(i)})^\top l^{(i)}-(\hat{w}_i)^\top \hat{B}R^{-1}C v_i+(S_{w,12}^{(i)})^\top (\tilde{X}^{(i-1)})^{-1}S_{v,12}^{(i)}=0,\label{small_x_r}\\
\big(A-K_v^{(i-1)}C\big) v_i-E v_is_v^{(i)}+B_{\perp}^{(i-1)}l^{(i)}=0,\label{small_v_r}\\
\big(\hat{A}-\hat{B}\hat{K}_w^{(i-1)}\big)^\top \hat{w}_i-\hat{E}^\top \hat{w}_is_w^{(i)}+(\hat{C}_{\perp}^{(i-1)})^\top l^{(i)}=0.\label{small_w_r}
\end{align}

With these definitions, one can verify that \(V^{(i)}\), \(\hat{W}^{(i)}\), and \((\tilde{X}^{(i)})^{-1}\) satisfy the following Sylvester equations:
\begin{align}
A V^{(i)} - E V^{(i)} S_v^{(i)} + B L^{(i)} &= 0, \label{Sylv_V} \\[4pt]
\hat{A}^\top \hat{W}^{(i)} - \hat{E}^\top \hat{W}^{(i)} S_w^{(i)} + \hat{C}^\top L^{(i)} &= 0, \label{Sylv_W} \\[4pt]
(S_w^{(i)})^\top (\tilde{X}^{(i)})^{-1} S_v^{(i)}- (\tilde{X}^{(i)})^{-1}-(L^{(i)})^\top L^{(i)}-(\hat{W}^{(i)})^\top \hat{B}R^{-1}C V^{(i)}&=0. \label{Sylv_X}
\end{align}
Because of the block triangular structure of \(S_v^{(i)}\) and \(S_w^{(i)}\), their eigenvalues are \((-\alpha_1,\dots,-\alpha_i)\) and \((-\beta_1,\dots,-\beta_i)\), respectively, each with multiplicity \(m\). By the relationship between Sylvester equations and rational Krylov subspaces established in \cite{gallivan2004sylvester}, the bases \(V^{(i)}\) and \(\hat{W}^{(i)}\) satisfy properties \eqref{span_prop1} and \eqref{span_prop2}. Hence, the interpolation conditions \eqref{int_cond} are satisfied.

Premultiplying \eqref{Sylv_V} and \eqref{Sylv_W} by \((W^{(i)})^\top\) and \((\hat{V}^{(i)})^\top\), respectively, gives
\[
A_r^{(i)} = S_v^{(i)} - B_r^{(i)} L^{(i)} \qquad \text{and} \qquad \hat{A}_r^{(i)} = (S_w^{(i)})^\top - (L^{(i)})^\top \hat{C}_r^{(i)}.
\]
Moreover, the interpolation conditions \eqref{int_cond} do not depend on the particular choices of \(W^{(i)}\) and \(\hat{V}^{(i)}\). Thus, the parameters \(B_r^{(i)}\) and \(\hat{C}_r^{(i)}\) may be chosen freely without affecting the interpolation conditions \eqref{int_cond}, provided that \((S_v^{(i)},L^{(i)})\) and \((S_w^{(i)},L^{(i)})\) are observable. For more details, see \cite{wolfthesis,panzerthesis,astolfi2010model}.

The next theorem gives particular choices of \(B_r^{(i)}\) and \(\hat{C}_r^{(i)}\) for which \(\tilde{X}^{(i)}\) is a stabilizing solution of the projected DTNARE \eqref{proj_dtnare}.
\begin{theorem}\label{th1}
Suppose that the ADI shifts \(\alpha_i\) and \(\beta_i\) lie outside the unit circle, that is, \(|\alpha_i|>1\) and \(|\beta_i|>1\). Let \(V^{(i)}\) and \(\hat{W}^{(i)}\) satisfy the Sylvester equations \eqref{Sylv_V} and \eqref{Sylv_W}, respectively, with \(s_v^{(i)}\), \(s_w^{(i)}\), \(l^{(i)}\), \(\tilde{x}_i\), \(v_i\), \(\hat{w}_i\), \(S_v^{(i)}\), \(S_w^{(i)}\), \(L^{(i)}\), \(S_{v,12}^{(i)}\), \(S_{w,12}^{(i)}\), \(B_{\perp}^{(i)}\), \(\hat{C}_{\perp}^{(i)}\), and \(\tilde{X}^{(i)}\) defined in \eqref{sl1}--\eqref{Sylv_X}. Assume that the pairs \((S_v^{(i)},L^{(i)})\) and \((S_w^{(i)},L^{(i)})\) are observable, and that the matrices \(\tilde{X}^{(i)}\) and \(R+C_r^{(i)} \tilde{X}^{(i)} \hat{B}_r^{(i)}\) are invertible. Suppose further that there exist matrices \(W^{(i)}\) and \(\hat{V}^{(i)}\) satisfying
\[
(W^{(i)})^\top E V^{(i)}=I,\qquad (\hat{W}^{(i)})^\top \hat{E}\hat{V}^{(i)}=I,
\]
so that
\begin{align}
A_r^{(i)} &= (W^{(i)})^\top A V^{(i)} = S_v^{(i)}-B_r^{(i)}L^{(i)},\nonumber\\
B_r^{(i)}&=(W^{(i)})^\top B,\quad C_r^{(i)} =CV^{(i)}=\begin{bmatrix}C_r^{(i-1)}&c_i\end{bmatrix},\nonumber\\
\hat{A}_r^{(i)} &= (\hat{W}^{(i)})^\top \hat{A} \hat{V}^{(i)} = (S_w^{(i)})^\top-(L^{(i)})^\top\hat{C}_r^{(i)},\nonumber\\
\hat{B}_r^{(i)}&=(\hat{W}^{(i)})^\top \hat{B}=\begin{bmatrix}\hat{B}_r^{(i-1)}\\\hat{b}_i\end{bmatrix}, \quad \hat{C}_r^{(i)} =\hat{C}\hat{V}^{(i)}.\label{expand_1}
\end{align}
If the free parameters \(B_r^{(i)}\) and \(\hat{C}_r^{(i)}\) are selected as
\begin{align}
B_r^{(i)}&=\begin{bmatrix}B_r^{(i-1)}\\b_i\end{bmatrix}=\tilde{X}^{(i)}(S_w^{(i)})^{-\top}(L^{(i)})^\top,\nonumber\\
\hat{C}_r^{(i)} &= \begin{bmatrix} \hat{C}_r^{(i-1)}& \hat{c}_i \end{bmatrix} = L^{(i)}\big( S_v^{(i)}\big)^{-1}\tilde{X}^{(i)},\label{free_par}
\end{align}
where
\begin{align}
b_i&=\tilde{x}_i(s_w^{(i)})^{-\top}\big((l^{(i)})^\top-(S_{w,12}^{(i)})^\top(\tilde{X}^{(i-1)})^{-1}B_r^{(i-1)}\big),\nonumber\\ \hat{c}_i&=\big(l^{(i)}-\hat{C}_r^{(i-1)}(\tilde{X}^{(i-1)})^{-1}S_{v,12}^{(i)}\big)(s_v^{(i)})^{-1}\tilde{x}_i,\label{bch}
\end{align}
then the following statements hold:
\begin{enumerate}
  \item \(\tilde{X}^{(i)}\) is a stabilizing solution of the projected DTNARE \eqref{proj_dtnare}. The closed-loop matrices
  \begin{align}
 A_{r,cl}^{(i)}&=A_r^{(i)}-A_r^{(i)}\tilde{X}^{(i)}\hat{B}_r^{(i)}\big(R+C_r^{(i)}\tilde{X}^{(i)}\hat{B}_r^{(i)}\big)^{-1}C_r^{(i)},\nonumber\\
 \hat{A}_{r,cl}^{(i)}&=\hat{A}_r^{(i)}-\hat{B}_r^{(i)}\big(R+C_r^{(i)}\tilde{X}^{(i)}\hat{B}_r^{(i)}\big)^{-1}C_r^{(i)}\tilde{X}^{(i)}\hat{A}_r^{(i)}
  \end{align}
  have eigenvalues \(-\frac{1}{\beta_1},\dots,-\frac{1}{\beta_i}\) and \(-\frac{1}{\alpha_1},\dots,-\frac{1}{\alpha_i}\), respectively, each with multiplicity \(m\).
  \item The residual \(R_x^{(i)}\), which satisfies the Petrov–Galerkin projection condition
  \[
  (W^{(i)})^\top R_x^{(i)} \hat{V}^{(i)} = 0,
  \]
  is given by
  \[
  R_x^{(i)}=B_{\perp}^{(i)}\big(M^{(i)}\big)^{-1} \hat{C}_{\perp}^{(i)},
  \]
  where
  \[
  M^{(i)}=I_{m}-\hat{C}_r^{(i)}(\tilde{X}^{(i)})^{-1}B_r^{(i)}=M^{(i-1)}-\hat{c}_i(\tilde{x}_i)^{-1}b_i,
  \]
  with \(M^{(0)}=I_{m}\).
  \item \(V^{(i)}\) and \(\hat{W}^{(i)}\) satisfy the Sylvester equations
  \begin{align}
    A V^{(i)} - E V^{(i)} A_r^{(i)} + B_{\perp}^{(i)} L^{(i)} &= 0,\label{Sylv_V2}\\
  \hat{A}^\top \hat{W}^{(i)} - \hat{E}^\top \hat{W}^{(i)} (\hat{A}_r^{(i)})^\top + (\hat{C}_{\perp}^{(i)})^\top L^{(i)} &= 0.\label{Sylv_W2}
  \end{align}
\end{enumerate}
\end{theorem}
\begin{proof}
The proof is provided in the Appendix.
\end{proof}
\subsection{Algorithm}
This subsection describes a recursive implementation of the low-rank solver developed in the previous subsection.

The matrices \(K_v^{(i)}\) and \(\hat{K}_w^{(i)}\) can be updated recursively as
\begin{align*}
K_v^{(i)}&=K_v^{(i-1)}+Ev_i\tilde{x}_i(s_w^{(i)})^{-\top}\Big(\hat{b}_i-\big(S_{w,12}^{(i)}\big)^\top \big(S_w^{(i-1)}\big)^{-\top}\hat{B}_{r}^{(i-1)}\Big)R^{-1},\\
\hat{K}_w^{(i)}&=\hat{K}_w^{(i-1)}+R^{-1}\Big(c_i-C_r^{(i-1)}\big(S_v^{(i-1)}\big)^{-1}S_{v,12}^{(i)}\Big)(s_v^{(i)})^{-1}\tilde{x}_i(\hat{w}_i)^\top \hat{E},
\end{align*}
with \(K_v^{(0)}=\mathbf{0}\) and \(\hat{K}_w^{(0)}=\mathbf{0}\).

The matrices \(B_{\perp}^{(i)}\) and \(\hat{C}_{\perp}^{(i)}\) can be updated recursively by
\[
B_{\perp}^{(i)}=B_{\perp}^{(i-1)}-Ev_ib_i\quad \text{and}\quad \hat{C}_{\perp}^{(i)}=\hat{C}_{\perp}^{(i-1)}-\hat{c}_i(\hat{w}_i)^\top \hat{E},
\]
with \(B_{\perp}^{(0)}=B\) and \(\hat{C}_{\perp}^{(0)}=\hat{C}\).

The new basis blocks \(v_i\) and \(\hat{w}_i\) are obtained from the following shifted linear systems.

For Case I,
\begin{align}
\Big(A-K_v^{(i-1)}C+\alpha_i E\Big)v_i=B_{\perp}^{(i-1)}\quad \text{and}\quad \Big(\hat{A}^\top-\big(\hat{K}_w^{(i-1)}\big)^\top\hat{B}^\top+\beta_i \hat{E}^\top\Big)\hat{w}_i=\big(\hat{C}_{\perp}^{(i-1)}\big)^\top.\label{vw_1}
\end{align}

For Case II, let
\[
v_i=\begin{bmatrix}\mathrm{Re}(\tilde{v}_i)&\mathrm{Im}(\tilde{v}_i)\end{bmatrix},\qquad
\hat{w}_i=\begin{bmatrix}\mathrm{Re}(\tilde{w}_i)&\mathrm{Im}(\tilde{w}_i)\end{bmatrix},
\]
where
\begin{align}
\Big(A-K_v^{(i-1)}C+\alpha_i E\Big)\tilde{v}_i=B_{\perp}^{(i-1)}\quad \text{and}\quad \Big(\hat{A}^\top-\big(\hat{K}_w^{(i-1)}\big)^\top\hat{B}^\top+\beta_i \hat{E}^\top\Big)\tilde{w}_i=\big(\hat{C}_{\perp}^{(i-1)}\big)^\top.\label{vw_2}
\end{align}

For Case III, let
\[
v_i=\begin{bmatrix}\mathrm{Re}(\tilde{v}_i)&\mathrm{Im}(\tilde{v}_i)\end{bmatrix},\qquad
\hat{w}_i=\begin{bmatrix}\tilde{w}_i&\tilde{w}_{i+1}\end{bmatrix},
\]
where
\begin{align}
\Big(A-K_v^{(i-1)}C+\alpha_i E\Big)\tilde{v}_i&=B_{\perp}^{(i-1)},\\
\Big(\hat{A}^\top-\big(\hat{K}_w^{(i-1)}\big)^\top\hat{B}^\top+\beta_i \hat{E}^\top\Big)\tilde{w}_i&=\big(\hat{C}_{\perp}^{(i-1)}\big)^\top,\\
\Big(\hat{A}^\top-\big(\hat{K}_w^{(i-1)}\big)^\top\hat{B}^\top+\beta_{i+1} \hat{E}^\top\Big)\tilde{w}_{i+1}&=\hat{E}^\top \tilde{w}_i.\label{vw_3}
\end{align}

For Case IV, let
\[
v_i=\begin{bmatrix}\tilde{v}_i&\tilde{v}_{i+1}\end{bmatrix},\qquad
\hat{w}_i=\begin{bmatrix}\mathrm{Re}(\tilde{w}_i)&\mathrm{Im}(\tilde{w}_i)\end{bmatrix},
\]
where
\begin{align}
\Big(A-K_v^{(i-1)}C+\alpha_i E\Big)\tilde{v}_i&=B_{\perp}^{(i-1)},\\
\Big(A-K_v^{(i-1)}C+\alpha_{i+1} E\Big)\tilde{v}_{i+1}&=E\tilde{v}_i,\\
\Big(\hat{A}^\top-\big(\hat{K}_w^{(i-1)}\big)^\top\hat{B}^\top+\beta_i \hat{E}^\top\Big)\tilde{w}_i&=\big(\hat{C}_{\perp}^{(i-1)}\big)^\top.\label{vw_4}
\end{align}
These shifted linear systems can be solved efficiently using the Sherman–Morrison–Woodbury (SMW) formula \cite{golub2013matrix}, as done in related works \cite{benner2018radi,zulfiqar2026ldl,zulfiqar2026d}.

The matrix
\[
G^{(i)}=R+CV^{(i)}\tilde{X}^{(i)}\big(\hat{W}^{(i)}\big)^\top \hat{B}
\]
can be updated recursively by
\[
G^{(i)}=G^{(i-1)}+c_i\tilde{x}_i\hat{b}_i,
\]
with \(G^{(0)}=R\).

The auxiliary matrices
\[
K_1^{(i)}=AV^{(i)}\tilde{X}^{(i)}\big(\hat{W}^{(i)}\big)^\top\hat{B},\qquad
\hat{K}_1^{(i)}=CV^{(i)}\tilde{X}^{(i)}\big(\hat{W}^{(i)}\big)^\top\hat{A}
\]
satisfy
\[
K_1^{(i)}=K_1^{(i-1)}+Av_i\tilde{x}_i\hat{b}_i\quad \text{and}\quad \hat{K}_1^{(i)}=\hat{K}_1^{(i-1)}+c_i\tilde{x}_i(\hat{w}_i)^\top\hat{A}. 
\]
The gain matrices \(K\) and \(\hat{K}\) can then be approximated by
\[
K\approx K_1^{(i)}\big(G^{(i)}\big)^{-1},\qquad \hat{K}\approx \big(G^{(i)}\big)^{-1}\hat{K}_1^{(i)}.
\]
Instead of solving the Stein equation \eqref{small_x_r} explicitly to obtain \((\tilde{x}_i)^{-1}\), it can be evaluated using the following analytical formulas.

For Case I,
\begin{align}
(\tilde{x}_i)^{-1}=\mu_i\mathcal{Q}^{(i)},\label{x_inv_case_1}
\end{align}
where 
\[
\mu_i=\frac{1}{1-\alpha_i\beta_i}\quad \text{and} \quad\mathcal{Q}^{(i)}=\big(S_{w,12}^{(i)}\big)^\top \big(\tilde{X}^{(i-1)}\big)^{-1}S_{v,12}^{(i)}-(l^{(i)})^\top l^{(i)}-\hat{b}_iR^{-1}c_i.
\]

For Case II,
\begin{align}
(\tilde{x}_i)^{-1}=\frac{1}{2}\mathrm{Re}\Bigg(\Big(\begin{bmatrix}\mu_i+\nu_i&-j(\nu_i-\mu_i)\\j(\nu_i-\mu_i)&\mu_i+\nu_i\end{bmatrix}\otimes I_m\Big) \mathcal{Q}^{(i)} \Big(\begin{bmatrix}1&-j\\j&1\end{bmatrix}\otimes I_m\Big)\Bigg),\label{x_inv_case_2}
\end{align}
where 
\[
\nu_i=\frac{1}{1-\alpha_i\overline{\beta}_{i}}.
\]

For Case III,
\begin{align}
(\tilde{x}_i)^{-1}=\mathrm{Re}\Bigg(\Big(\begin{bmatrix}\mu_i&0\\-\alpha_i\mu_i\gamma_i&\gamma_i\end{bmatrix}\otimes I_m\Big) \mathcal{Q}^{(i)} \Big(\begin{bmatrix}1&-j\\j&1\end{bmatrix}\otimes I_m\Big)\Bigg),\label{x_inv_case_3}
\end{align}
where
\[
\gamma_i=\frac{1}{1-\alpha_i\beta_{i+1}}.
\]

For Case IV,
\begin{align}
(\tilde{x}_i)^{-1}=\mathrm{Re}\Bigg(\Big(\begin{bmatrix}1&j\\-j&1\end{bmatrix}\otimes I_m\Big) \mathcal{Q}^{(i)} \Big(\begin{bmatrix}\mu_i&-\beta_i\mu_i\zeta_i\\0&\zeta_i\end{bmatrix}\otimes I_m\Big)\Bigg),\label{x_inv_case_4}
\end{align}
where
\[
\zeta_i=\frac{1}{1-\alpha_{i+1}\beta_{i}}.
\]
\begin{algorithm}[!h]
\caption{DTNS-RADI}\label{alg1}
\DontPrintSemicolon
\KwIn{
  Matrices of DTNARE \eqref{dtnare}: $E$, $A$, $B$, $C$, $\hat{E}$, $\hat{A}$, $\hat{B}$, $\hat{C}$, $R$; ADI shifts: $\{\alpha_i\}_{i=1}^k$ and $\{\beta_i\}_{i=1}^k$ satisfying $|\alpha_i|>1$ and $|\beta_i|>1$; Tolerance: $\tau\in[0,1]$.
}
\KwOut{
  Approximation of $X$: $X \approx\bar{X}^{(i)}= V^{(i)}\tilde{X}^{(i)}(\hat{W}^{(i)})^\top$; Approximation of gain matrices: $K\approx A\bar{X}^{(i)}\hat{B}(R+C\bar{X}^{(i)}\hat{B})^{-1}$ and $\hat{K}\approx(R+C\bar{X}^{(i)}\hat{B})^{-1}C\bar{X}^{(i)}\hat{A}$; Residual $R_x^{(i)}$: $R_x^{(i)}=B_{\perp}^{(i)} \big(M^{(i)}\big)^{-1} \hat{C}_{\perp}^{(i)}$.
}

\BlankLine
\textbf{Initialization:}
$B_{\perp}^{(0)}=B$, $\hat{C}_{\perp}^{(0)} = \hat{C}$, $V^{(0)}=[\;]$, $\hat{W}^{(0)} = [\;]$, $\tilde{X}^{(0)} = [\;]$, $S_v^{(0)}=[\;]$, $S_w^{(0)}=[\;]$, $L^{(0)}=[\;]$, $\big(S_v^{(0)}\big)^{-1}=[\;]$, $S_{v,12}^{(0)}=[\;]$, $\big(S_w^{(0)}\big)^{-1}=[\;]$, $S_{w,12}^{(0)}=[\;]$, $K_v^{(0)} = \mathbf{0}$, $\hat{K}_w^{(0)} = \mathbf{0}$, $G^{(0)}=R$, $K_1^{(0)}=\mathbf{0}$, $\hat{K}_1^{(0)}=\mathbf{0}$, $M^{(0)}=I_{m}$, $B_r^{(0)}=[\;]$, $C_r^{(0)}=[\;]$, $\hat{B}_r^{(0)}=[\;]$, $\hat{C}_r^{(0)}=[\;]$, $i = 1$.

\BlankLine
\While{$\dfrac{\big\|B_{\perp}^{(i-1)}\big(M^{(i-1)}\big)^{-1} \hat{C}_{\perp}^{(i-1)}\big\|}{\big\|B\hat{C}\big\|} \geq \tau$}{
 Compute \(v_i\) and \(\hat{w}_i\) from \eqref{vw_1}--\eqref{vw_4} according to the case the shifts \((\alpha_i,\beta_i)\) belong to.\label{step1}\\
    Set $s_v^{(i)}$, $s_w^{(i)}$, and $l^{(i)}$ from \eqref{sl1}--\eqref{sl4} according to the case the shifts \((\alpha_i,\beta_i)\) belong to.\\
    Compute $\hat{b}_i=\hat{w}_i^\top \hat{B}$ and $c_i=Cv_i$, and set $S_{v,12}^{(i)}=\tilde{X}^{(i-1)}(S_w^{(i-1)})^{-\top}\big((L^{(i-1)})^\top l^{(i)}+\hat{B}_r^{(i-1)}R^{-1}c_i\big)$ and $S_{w,12}^{(i)}=(\tilde{X}^{(i-1)})^\top(S_v^{(i-1)})^{-\top}\big((L^{(i-1)})^\top l^{(i)}+(C_r^{(i-1)})^\top R^{-\top}\hat{b}_i^\top\big)$.\\ 
    Compute $(\tilde{x}_i)^{-1}$ from \eqref{x_inv_case_1}--\eqref{x_inv_case_4} according to the case the shifts \((\alpha_i,\beta_i)\) belong to.\\
    Set $b_i$ and $\hat{c}_i$ as in \eqref{bch}, and expand $S_v^{(i)}$, $(S_v^{(i)})^{-1}$, $S_w^{(i)}$, $(S_w^{(i)})^{-1}$, $L^{(i)}$, $\tilde{X}^{(i)}$, $V^{(i)}$, $\hat{W}^{(i)}$, $B_r^{(i)}$, $C_r^{(i)}$, $\hat{B}_r^{(i)}$, and $\hat{C}_r^{(i)}$ as in \eqref{big_S_L}, \eqref{expand_1}, and \eqref{free_par}.\\
    Update $B_{\perp}^{(i)}=B_{\perp}^{(i-1)}-Ev_ib_i$, $\hat{C}_{\perp}^{(i)}=\hat{C}_{\perp}^{(i-1)}-\hat{c}_i(\hat{w}_i)^\top\hat{E}$, $M^{(i)}=M^{(i-1)}-\hat{c}_i(\tilde{x}_i)^{-1}b_i$, $K_v^{(i)}=K_v^{(i-1)}+Ev_i\tilde{x}_i(s_w^{(i)})^{-\top}\big(\hat{b}_i-(S_{w,12}^{(i)})^\top (S_w^{(i-1)})^{-\top}\hat{B}_{r}^{(i-1)}\big)R^{-1}$, $\hat{K}_w^{(i)}=\hat{K}_w^{(i-1)}+R^{-1}\big(c_i-C_r^{(i-1)}(S_v^{(i-1)})^{-1}S_{v,12}^{(i)}\big)(s_v^{(i)})^{-1}\tilde{x}_i(\hat{w}_i)^\top \hat{E}$, $G^{(i)}=G^{(i-1)}+c_i\tilde{x}_i\hat{b}_i$, $K_1^{(i)}=K_1^{(i-1)}+Av_i\tilde{x}_i\hat{b}_i$, $\hat{K}_1^{(i)}=\hat{K}_1^{(i-1)}+c_i\tilde{x}_i(\hat{w}_i)^\top\hat{A}$, and $i\gets i+1$.
}
Set $K\approx K_1^{(i)}\big(G^{(i)}\big)^{-1}$ and $\hat{K}\approx \big(G^{(i)}\big)^{-1}\hat{K}_1^{(i)}$.
\end{algorithm}
\subsection{Automatic Shift Generation}
DTNS-RADI can be viewed primarily as a recursive interpolation-based model order reduction (MOR) method that, as a byproduct, also provides an approximation to the DTNARE \eqref{dtnare}. Error analysis for MOR methods whose bases satisfy Sylvester equations of the forms \eqref{Sylv_V}, \eqref{Sylv_V2}, \eqref{Sylv_W}, and \eqref{Sylv_W2} is developed in \cite{wolfthesis} and \cite{panzerthesis}; the present subsection follows that analysis.

Define
\begin{align}
G_{\perp}^{(i)}(z)=C(zE-A)^{-1}B_{\perp}^{(i)}\quad \text{and}\quad \hat{G}_{\perp}^{(i)}(z)=\hat{C}_{\perp}^{(i)}(z\hat{E}-\hat{A})^{-1}\hat{B}.
\end{align}
The approximation errors \(G(z)-G_r^{(i)}(z)\) and \(\hat{G}(z)-\hat{G}_r^{(i)}(z)\) can then be factorized as
\begin{align}
G(z)-G_r^{(i)}(z)&=G_{\perp}^{(i)}(z)\Big(-L^{(i)}\big(zI-A_r^{(i)}\big)^{-1}B_r^{(i)}+I\Big),\\
\hat{G}(z)-\hat{G}_r^{(i)}(z)&=\Big(-\hat{C}_r^{(i)}\big(zI-\hat{A}_r^{(i)}\big)^{-1}(L^{(i)})^\top+I\Big)\hat{G}_{\perp}^{(i)}(z).
\end{align}
Lemma 3.3 of \cite{wolfthesis} shows that these errors vanish when all poles of the realization \((E,A,B_{\perp}^{(i)},C)\) are uncontrollable and all poles of the realization \((\hat{E},\hat{A},\hat{B},\hat{C}_{\perp}^{(i)})\) are unobservable. Thus, DTNS-RADI can be interpreted as a recursive interpolation algorithm that simultaneously addresses two MOR problems. The peaks in the frequency response of \(G(z)\) associated with the most controllable modes, when captured by \(G_r^{(i)}(z)\), are flattened out in \(G_{\perp}^{(i)}(z)\). Although \(G(z)\) and \(G_{\perp}^{(i)}(z)\) have the same poles, once the most controllable poles of \(G(z)\) are flattened out in \(G_{\perp}^{(i)}(z)\), they become poorly controllable in \(G_{\perp}^{(i)}(z)\), and the error \(\|G(z)-G_r^{(i)}(z)\|\) drops significantly.

The \(\mathcal{H}_2\) norm error \(\|G(z)-G_r^{(i)}(z)\|_{\mathcal{H}_2}\) is dominated by the most controllable and observable poles of \(G(z)\). If \(G_r^{(i)}(z)\) interpolates \(G(z)\) at the reciprocals of these poles, the error decreases rapidly \cite{bunse2010h2,gugercin2008h_2}, since the dominant part of the \(\mathcal{H}_2\) norm error \(\|G(z)-G_r^{(i)}(z)\|_{\mathcal{H}_2}\) vanishes with this interpolation. Therefore, the mirror images of the reciprocals of the dominant poles of \(G_{\perp}^{(i)}(z)\) can be used as ADI shifts \(\alpha_i\) in DTNS-RADI to approximate \(G(z)\) accurately, since DTNS-RADI interpolates at the mirror images of the ADI shifts. In the same way, the mirror images of the reciprocals of the dominant poles of \(\hat{G}_{\perp}^{(i)}(z)\) can be used as ADI shifts \(\beta_i\) in DTNS-RADI to approximate \(\hat{G}(z)\) accurately.

In large-scale settings, directly computing the poles of \(G_{\perp}^{(i)}(z)\) and \(\hat{G}_{\perp}^{(i)}(z)\) is infeasible. Instead, Ritz values of \(E^{-1}A\) and \(\hat{A}\hat{E}^{-1}\) can be obtained by projecting onto the interpolation bases accumulated during the DTNS-RADI iterations.

Set
\[
V_{\mathrm{ritz}} = \mathrm{orth}\Big(\begin{bmatrix} v_1 & \cdots & v_i \end{bmatrix}\Big)
\]
with implicit restarting: if the number of columns of \(V_{\mathrm{ritz}}\) exceeds a prescribed limit, the older vectors are discarded and a new basis is accumulated. Restarting mechanisms of this type are commonly used in eigenvalue algorithms to control the basis dimension \cite{saad2011numerical}.

Define the projected matrices
\[
E_{\mathrm{ritz}} = V_{\mathrm{ritz}}^\top E V_{\mathrm{ritz}}, \quad
A_{\mathrm{ritz}} = V_{\mathrm{ritz}}^\top A V_{\mathrm{ritz}}, \quad
B_{\mathrm{ritz}} = V_{\mathrm{ritz}}^\top B_{\perp}^{(i)}.
\]  
Compute the eigenvalue decomposition of \(E_{\mathrm{ritz}}^{-1}A_{\mathrm{ritz}}\) as
\[
E_{\mathrm{ritz}}^{-1}A_{\mathrm{ritz}} = T\,\mathrm{diag}\big(\lambda_1,\dots,\lambda_k\big)T^{-1}.
\]
Define
\[
r_{b,l} = T^{-1}(l,:)\,E_{\mathrm{ritz}}^{-1}B_{\mathrm{ritz}}.
\]
The most controllable pole \(\lambda_{\mathrm{dom}}\) of \(E_{\mathrm{ritz}}^{-1}A_{\mathrm{ritz}}\) is taken as the pole \(\lambda_l\) corresponding to the largest normalized residue, a standard measure in modal analysis \cite{gawronski2004dynamics},
\[
\phi_l = \frac{\|r_{b,l}\|_2^2}{1-|\lambda_l|^2}.
\]  
The next shift is then chosen as
\[
\alpha_{i+1}=-\frac{1}{\lambda_{\mathrm{dom}}}.
\]
If all poles \(\lambda_l\) lie outside or on the unit circle, the previous shift is repeated.

Similarly, set
\[
W_{\mathrm{ritz}} = \mathrm{orth}\Big(\begin{bmatrix} \hat{w}_1 & \cdots & \hat{w}_i \end{bmatrix}\Big)
\]
with implicit restarting. Define the projected matrices
\[
\hat{E}_{\mathrm{ritz}} = W_{\mathrm{ritz}}^\top \hat{E} W_{\mathrm{ritz}}, \quad
\hat{A}_{\mathrm{ritz}} = W_{\mathrm{ritz}}^\top \hat{A} W_{\mathrm{ritz}}, \quad
\hat{C}_{\mathrm{ritz}} = \hat{C}_{\perp}^{(i)} W_{\mathrm{ritz}}.
\]
Compute the eigenvalue decomposition of \(\hat{A}_{\mathrm{ritz}}\hat{E}_{\mathrm{ritz}}^{-1}\):
\[
\hat{A}_{\mathrm{ritz}}\hat{E}_{\mathrm{ritz}}^{-1} = \hat{T}\,\mathrm{diag}\big(\hat{\lambda}_1,\dots,\hat{\lambda}_r\big)\hat{T}^{-1}.
\]
Define
\[
r_{c,l} = \hat{C}_{\mathrm{ritz}}\hat{E}_{\mathrm{ritz}}^{-1} \hat{T}(:,l).
\]
The most observable pole \(\hat{\lambda}_{\mathrm{dom}}\) of \(\hat{A}_{\mathrm{ritz}}\hat{E}_{\mathrm{ritz}}^{-1}\) is taken as the pole \(\hat{\lambda}_l\) corresponding to the largest normalized residue
\[
\hat{\phi}_l = \frac{\|r_{c,l}\|_2^2}{1-|\hat{\lambda}_l|^2}.
\]
The next shift is then chosen as
\[
\beta_{i+1}=-\frac{1}{\hat{\lambda}_{\mathrm{dom}}}.
\]
If all poles \(\hat{\lambda}_l\) lie outside or on the unit circle, the previous shift is repeated.

One difficulty with this shift-generation strategy is that the generated shifts \(\alpha_i\) and \(\beta_i\) often cannot be grouped into the four cases used to keep \(V^{(i)}\), \(\hat{W}^{(i)}\), and \(\tilde{X}^{(i)}\) real-valued. This issue can be avoided by enforcing \(\alpha_i=\beta_i\) and constructing \(V_{\mathrm{ritz}}\) and \(W_{\mathrm{ritz}}\) in alternating iterations. That is, in one iteration \(\alpha_i\) is generated and \(\beta_i\) is set equal to \(\alpha_i\), while in the next iteration \(\beta_i\) is generated and \(\alpha_i\) is set equal to \(\beta_i\). As shown in the numerical section, the choice \(\alpha_i=\beta_i\) generally gives a better approximation than \(\alpha_i\neq \beta_i\). The same behavior was observed in \cite{zulfiqar2025unified} and \cite{zulfiqar2026low} for low-rank ADI algorithms for the Sylvester equation and the CTNARE.
\section{Numerical Results}
This section assesses the numerical performance of DTNS-RADI on two large-scale DTNAREs. MATLAB codes for reproducing the results reported here are publicly available at \cite{mycode}. For the automatic shift generation, the initial shifts are taken as $\alpha_1=\beta_1=2$. The maximum number of columns retained in $V_{\mathrm{ritz}}$ and $W_{\mathrm{ritz}}$ is limited to $10$. The basis $V_{\mathrm{ritz}}$ is constructed when $i$ is odd, while $W_{\mathrm{ritz}}$ is constructed when $i$ is even. Throughout the iterations, the ADI shifts satisfy $\alpha_i=\beta_i$ when automatic shift generation is used. The tolerance is set to $\tau=10^{-8}$. All experiments are carried out in MATLAB R2025b on a Windows 11 laptop equipped with 32 GB RAM and an Intel(R) Core(TM) Ultra 9 285H processor running at 2.9 GHz.
\subsection{Example 1}
The matrices $E\in\mathbb{R}^{10^7\times 10^7}$, $A\in\mathbb{R}^{10^7\times 10^7}$, $B\in\mathbb{R}^{10^7\times 1}$, and $C\in\mathbb{R}^{1\times10^7}$ are the state-space matrices of a discrete-time procedural model generated following the procedure described in \cite{zulfiqar2026}. The pole locations of $E^{-1}A$ are displayed in Figure \ref{fig1}. Of the $20$ complex poles, $10$ are dominant. The remaining complex and real poles are weakly controllable and weakly observable; further details can be found in \cite{zulfiqar2026}.
\begin{figure}[!h]
  \centering
  \includegraphics[width=8cm]{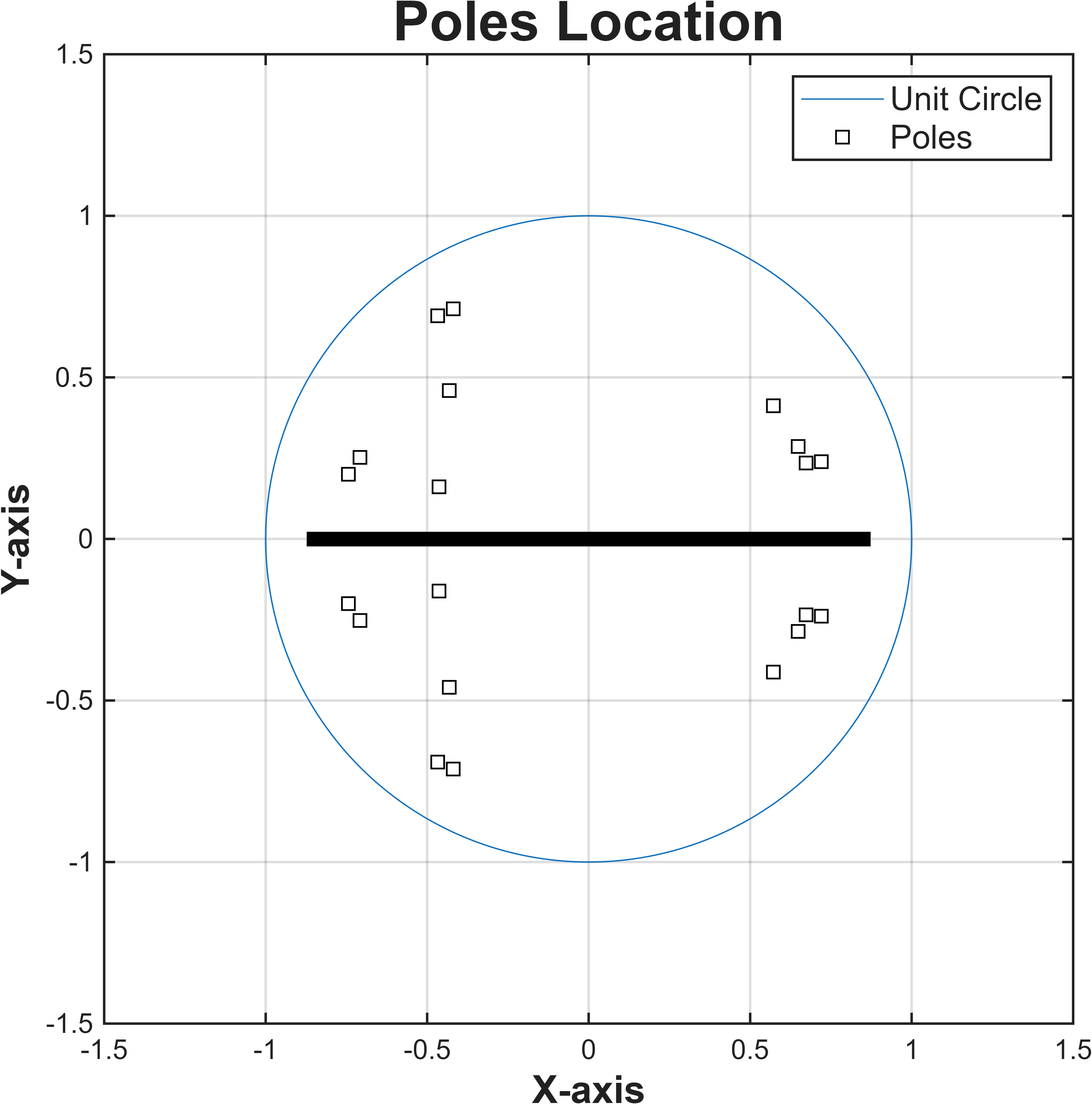}
  \caption{Pole locations of $E^{-1}A$}
  \label{fig1}
\end{figure}

The matrices $\hat{E}\in\mathbb{R}^{10^6\times 10^6}$, $\hat{A}\in\mathbb{R}^{10^6\times 10^6}$, $\hat{B}\in\mathbb{R}^{10^6\times 1}$, and $\hat{C}\in\mathbb{R}^{1\times10^6}$ are the state-space matrices of another discrete-time procedural model generated as described in \cite{zulfiqar2026}. The pole locations of $\hat{A}\hat{E}^{-1}$ are shown in Figure \ref{fig2}. Of the $12$ complex poles, $6$ are dominant. The remaining complex and real poles are weakly controllable and weakly observable \cite{zulfiqar2026}.
\begin{figure}[!h]
  \centering
  \includegraphics[width=8cm]{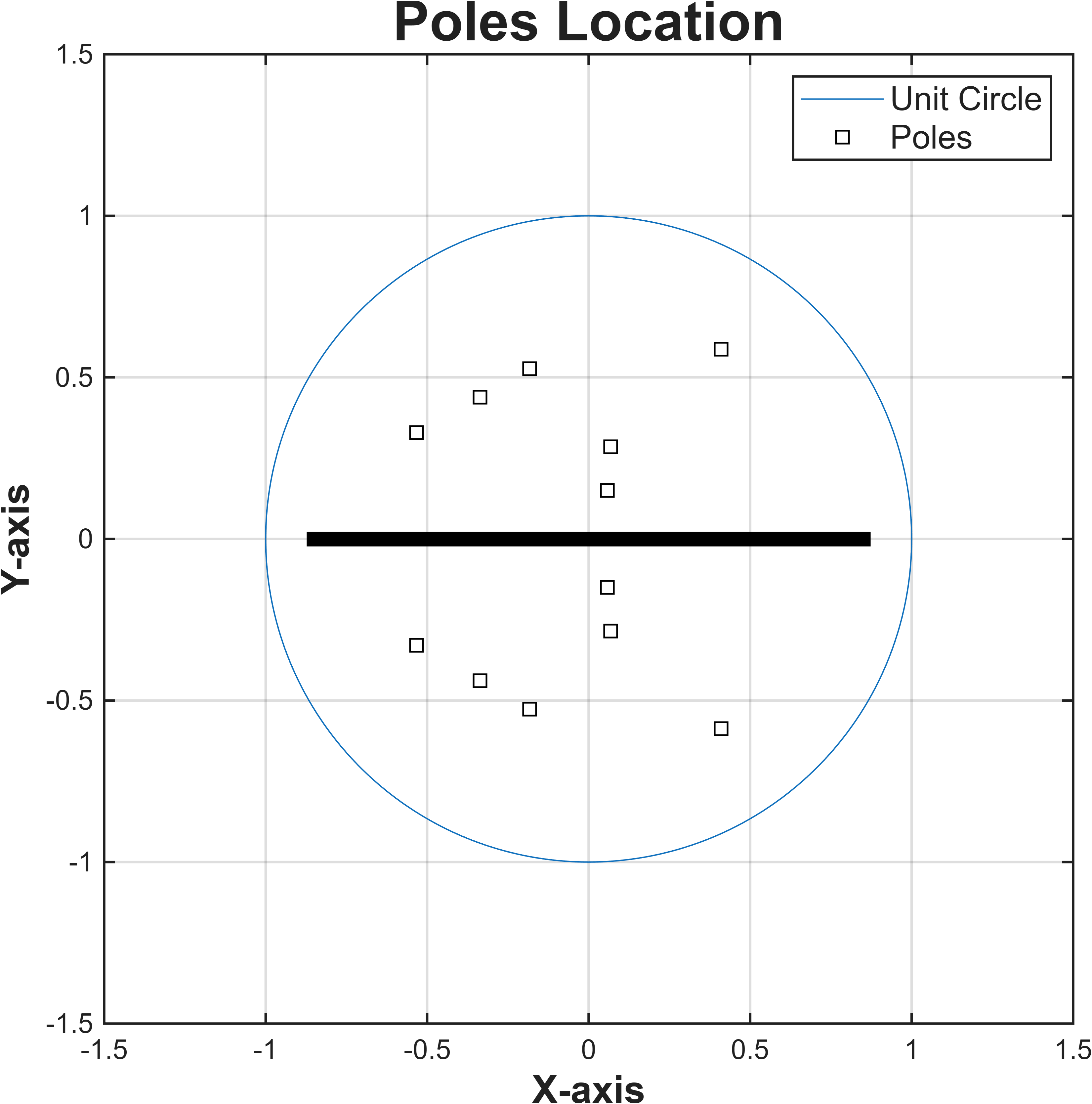}
  \caption{Pole locations of $\hat{A}\hat{E}^{-1}$}
  \label{fig2}
\end{figure}

Four different strategies for selecting the ADI shifts are examined. The first strategy employs the automatic shift generation method introduced in the previous section. The second strategy, referred to as ``Pole-I'', takes the mirror images of the reciprocals of the $50$ poles (including $20$ complex poles) of $E^{-1}A$ as ADI shifts with $\alpha_i=\beta_i$. The third strategy, ``Pole-II'', takes the mirror images of the reciprocals of the $50$ poles (including $12$ complex poles) of $\hat{A}\hat{E}^{-1}$ as ADI shifts with $\alpha_i=\beta_i$. The fourth strategy, ``Pole-III'', assigns the mirror images of the reciprocals of the $50$ poles (including $20$ complex poles) of $E^{-1}A$ to $\alpha_i$ and the mirror images of the reciprocals of the $50$ poles (including $12$ complex poles) of $\hat{A}\hat{E}^{-1}$ to $\beta_i$, i.e., $\alpha_i\neq \beta_i$. The spectral norm of the normalized residual
$\dfrac{\|R_x^{(i)}\|_2}{\|B\hat{C}\|_2}$ is plotted in Figure \ref{fig3}.
\begin{figure}[!h]
  \centering
  \includegraphics[width=10cm]{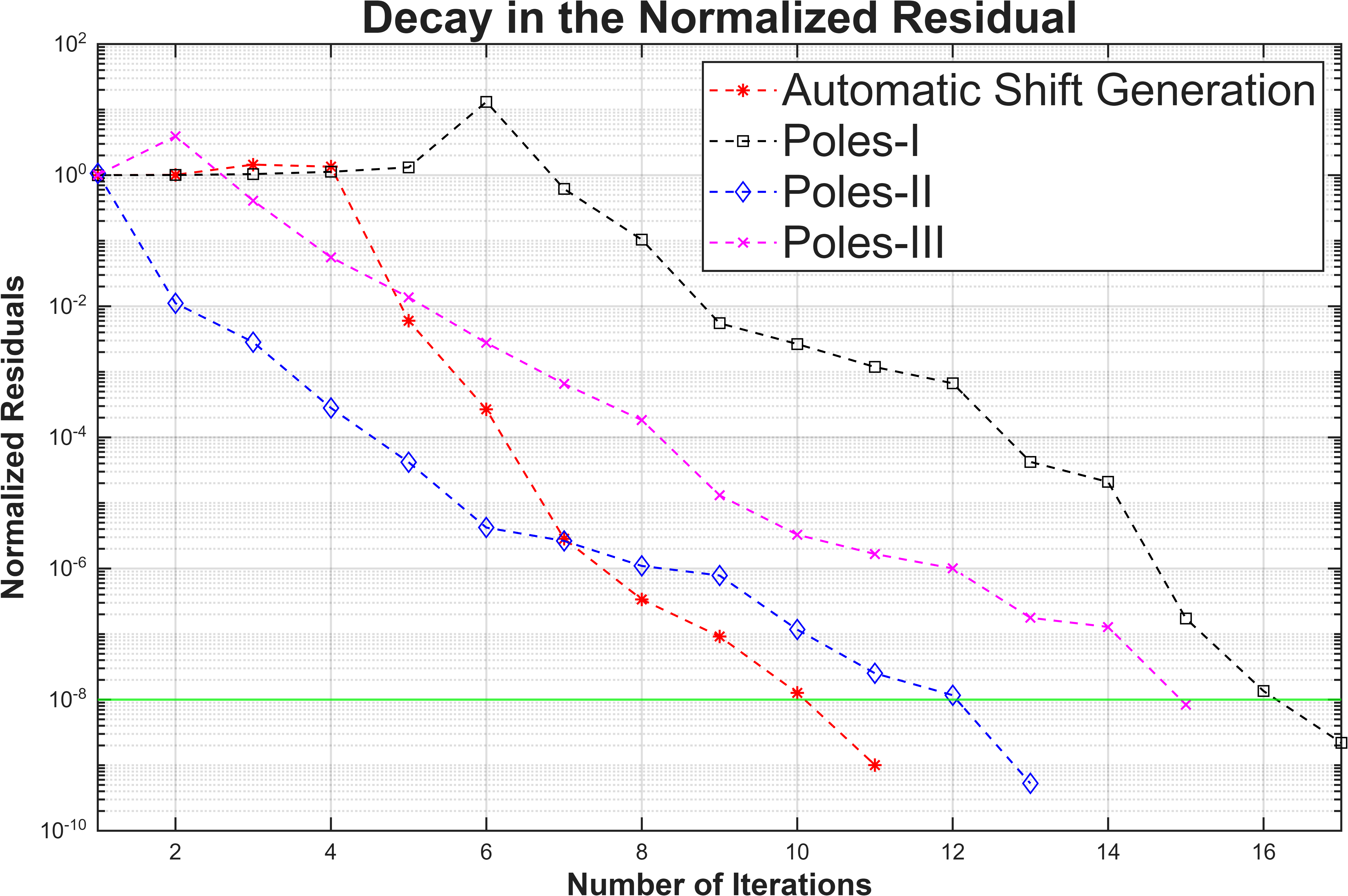}
  \caption{Decay in the normalized residual}
  \label{fig3}
\end{figure}

The automatic shift generation strategy outperforms the other three strategies. The computational times for the four strategies are reported in Table \ref{tab1}. Despite the large scale of the DTNARE matrices, the automatic shift generation produces shifts efficiently. Because it converged quickly, it took less time to execute, even though the other strategy used precomputed shifts and therefore spent no computational time generating shifts.
\begin{table}[!h]
\centering
\caption{Simulation Time Comparison}\label{tab1}
\begin{tabular}{|c|c|}\hline
Shift Generation Strategy & Elapsed Time (sec)\\\hline
Automatic Shift Generation& $67.3347$\\
Pole-I&$84.1151$\\
Pole-II& $71.2687$\\
Pole-III& $73.2912$\\
\hline
\end{tabular}
\end{table}

\subsection{Example 2}

The matrices $E\in\mathbb{R}^{10^6\times 10^6}$, $A\in\mathbb{R}^{10^6\times 10^6}$, $B\in\mathbb{R}^{10^6\times 2}$, and $C\in\mathbb{R}^{3\times10^6}$ are the state-space matrices of a discrete-time procedural model generated following the procedure described in \cite{zulfiqar2026}. The pole locations of $E^{-1}A$ are displayed in Figure \ref{fig4}. Of the $50$ complex poles, $20$ are dominant. The remaining complex and real poles are weakly controllable and weakly observable; further details can be found in \cite{zulfiqar2026}.
\begin{figure}[!h]
  \centering
  \includegraphics[width=8cm]{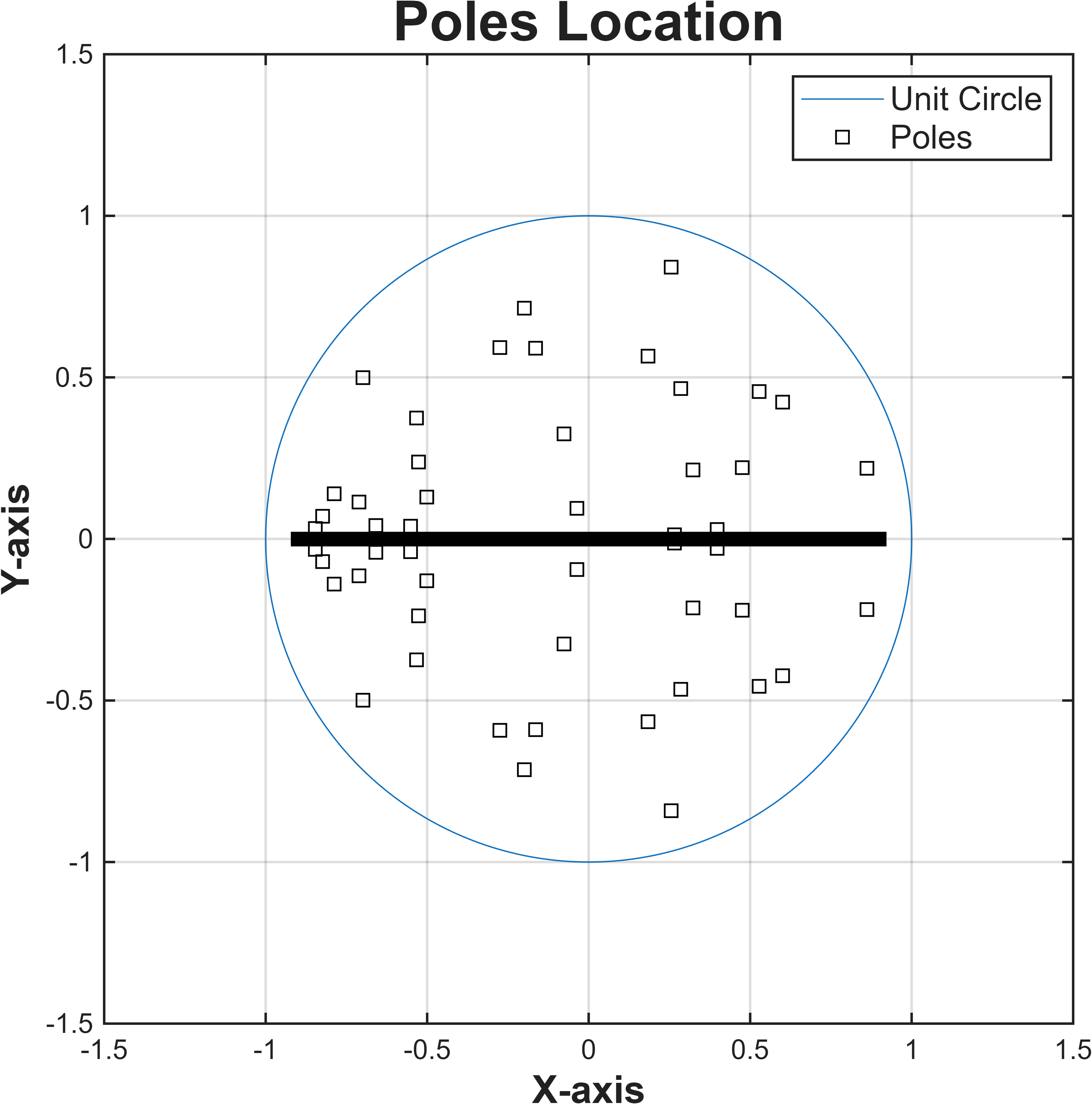}
  \caption{Pole locations of $E^{-1}A$}
  \label{fig4}
\end{figure}

The matrices $\hat{E}\in\mathbb{R}^{10^5\times 10^5}$, $\hat{A}\in\mathbb{R}^{10^5\times 10^5}$, $\hat{B}\in\mathbb{R}^{10^5\times 3}$, and $\hat{C}\in\mathbb{R}^{2\times10^5}$ are the state-space matrices of another discrete-time procedural model generated as described in \cite{zulfiqar2026}. The pole locations of $\hat{A}\hat{E}^{-1}$ are shown in Figure \ref{fig5}. Of the $50$ complex poles, $20$ are dominant. The remaining complex and real poles are weakly controllable and weakly observable \cite{zulfiqar2026}.
\begin{figure}[!h]
  \centering
  \includegraphics[width=8cm]{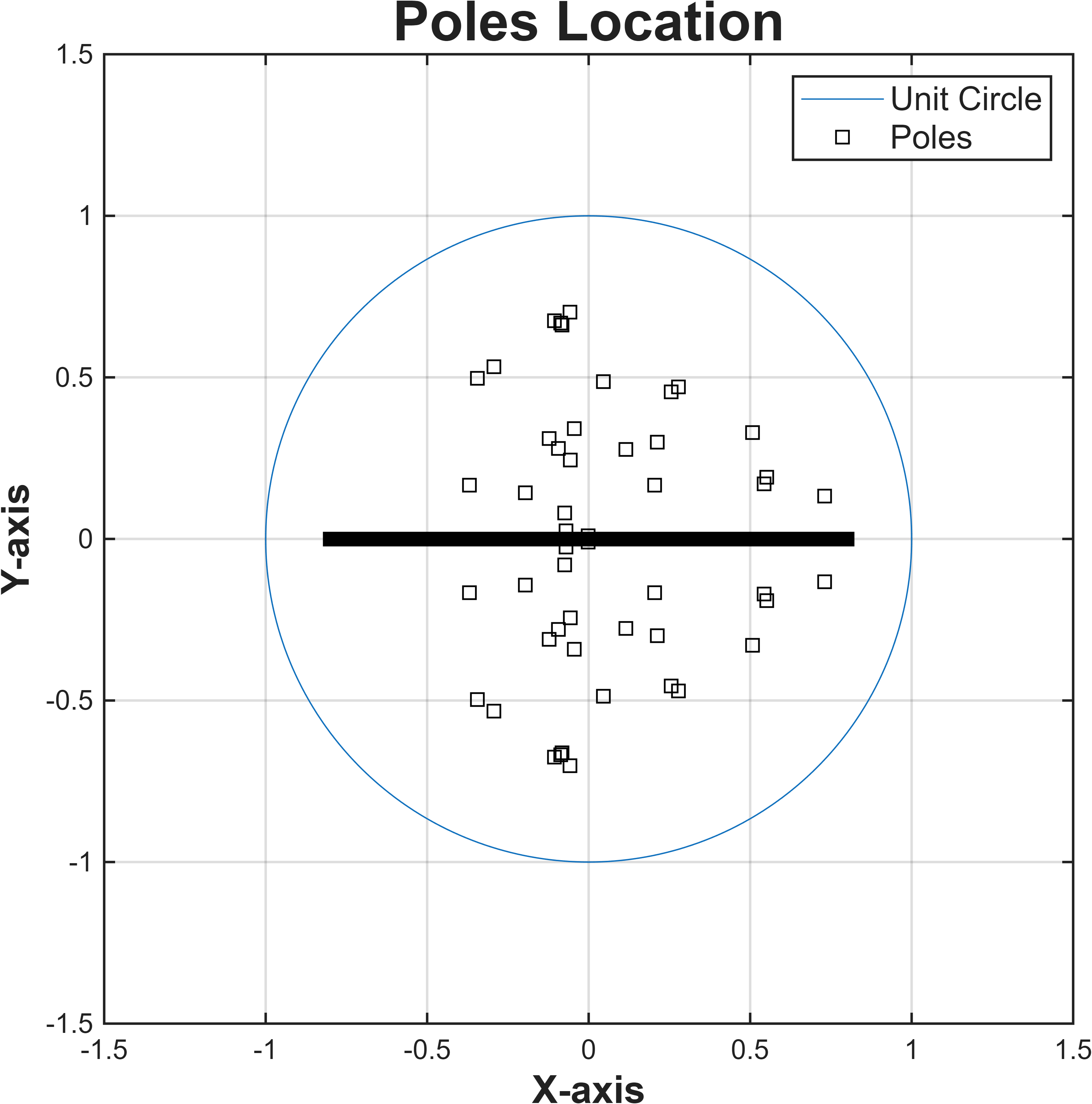}
  \caption{Pole locations of $\hat{A}\hat{E}^{-1}$}
  \label{fig5}
\end{figure}

Four different strategies for selecting the ADI shifts are examined. The first strategy employs the automatic shift generation method introduced in the previous section. The second strategy, ``Pole-I'', takes the mirror images of the reciprocals of the $100$ poles (including $50$ complex poles) of $E^{-1}A$ as ADI shifts with $\alpha_i=\beta_i$. The third strategy, ``Pole-II'', takes the mirror images of the reciprocals of the $100$ poles (including $50$ complex poles) of $\hat{A}\hat{E}^{-1}$ as ADI shifts with $\alpha_i=\beta_i$. The fourth strategy, ``Pole-III'', assigns the mirror images of the reciprocals of the $100$ poles (including $50$ complex poles) of $E^{-1}A$ to $\alpha_i$ and the mirror images of the reciprocals of the $100$ poles (including $50$ complex poles) of $\hat{A}\hat{E}^{-1}$ to $\beta_i$, i.e., $\alpha_i\neq \beta_i$. The spectral norm of the normalized residual $\dfrac{\|R_x^{(i)}\|_2}{\|B\hat{C}\|_2}$ is plotted in Figure \ref{fig6}.
\begin{figure}[!h]
  \centering
  \includegraphics[width=10cm]{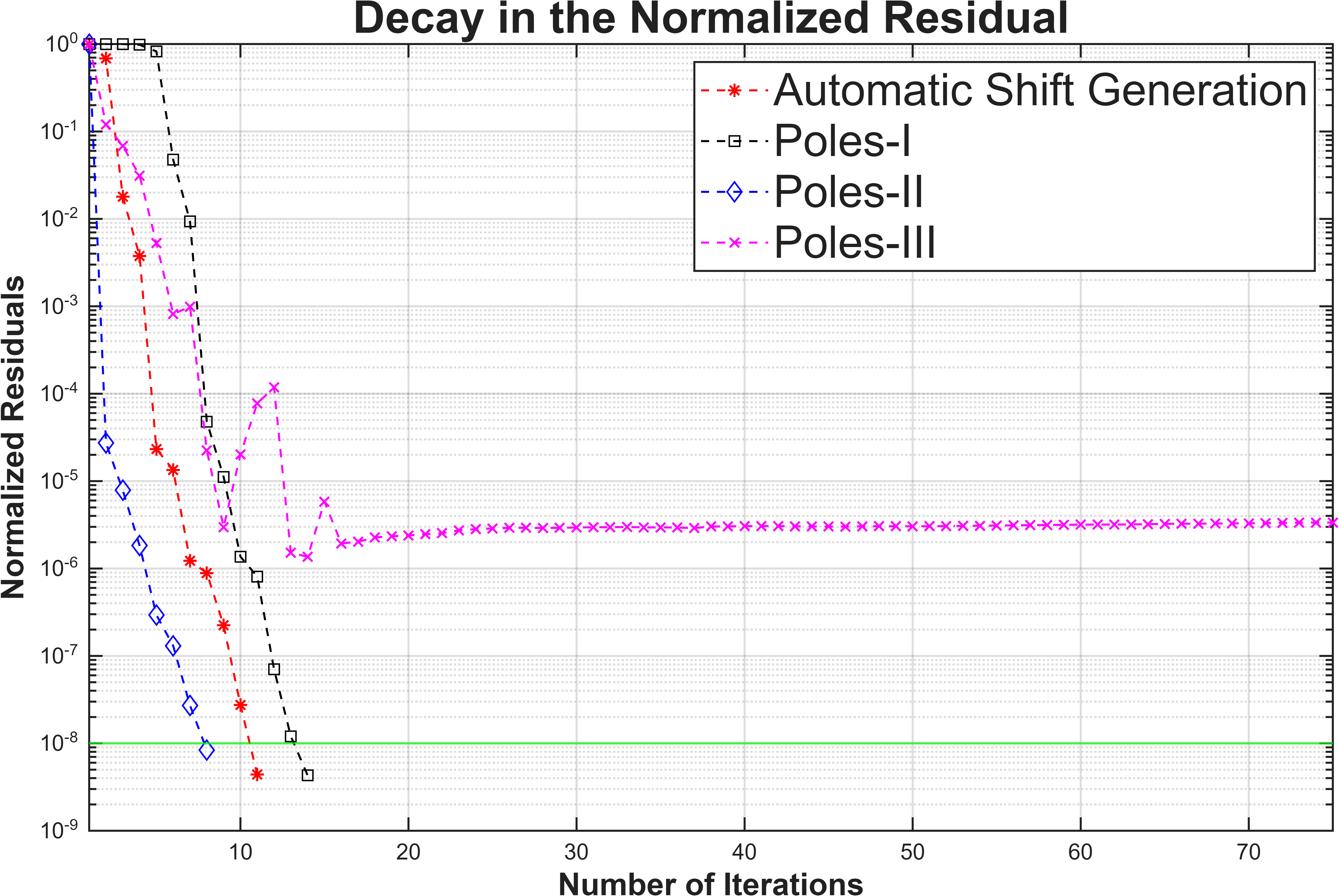}
  \caption{Decay in the normalized residual}
  \label{fig6}
\end{figure}

The automatic shift generation strategy performs well in this example. The Pole-III strategy fails to converge, even though it uses the dominant poles of both $E^{-1}A$ and $\hat{A}\hat{E}^{-1}$. This behavior is consistent with observations made for the Sylvester equation in \cite{zulfiqar2025unified}, where enforcing $\alpha_i=\beta_i$ has been shown to yield significantly better performance in low-rank ADI methods owing to improved interpolatory properties. The computational times for the four strategies are reported in Table \ref{tab2}. Despite the large scale of the DTNARE matrices, the automatic shift generation produces shifts efficiently. Pole-I and Pole-II use precomputed shifts, so they require less execution time. Pole-III, however, did not converge and therefore took considerably more time than the other strategies.
\begin{table}[!h]
\centering
\caption{Simulation Time Comparison}\label{tab2}
\begin{tabular}{|c|c|}\hline
Shift Generation Strategy & Elapsed Time (sec)\\\hline
Automatic Shift Generation& $7.6246$\\
Pole-I&$6.5548$\\
Pole-II& $3.6603$\\
Pole-III& $49.9795$\\
\hline
\end{tabular}
\end{table}
\section{Conclusion}
A low-rank ADI algorithm, DTNS-RADI, has been developed for solving large-scale DTNAREs that admit low-rank solutions. The method is based on the interpretation of low-rank ADI iterations as recursive interpolation schemes. Although the algorithm is projection-based, its pole-placement property guarantees that the projected DTNARE solved at each iteration admits a stabilizing solution. An automatic shift-generation strategy has also been introduced, making the method fully autonomous: once an initial shift is provided, subsequent shifts are generated without user intervention until the residual satisfies a prescribed tolerance. Numerical experiments on two large-scale DTNAREs demonstrate that DTNS-RADI is accurate and efficient, and that the proposed automatic shift strategy outperforms pole-based strategies using pre-computed shifts. These results show that DTNS-RADI is a practical and autonomous solver for computing low-rank stabilizing solutions of large-scale DTNAREs.
\section*{Appendix}

\begin{proof}
First, rearranging \eqref{Sylv_X} gives
\begin{equation}
(S_w^{(i)})^\top(\tilde X^{(i)})^{-1}S_v^{(i)}
-(L^{(i)})^\top L^{(i)}
=(\tilde X^{(i)})^{-1}
+\hat B_r^{(i)}R^{-1}C_r^{(i)}.
\label{eq:SXrearr}
\end{equation}
Premultiplying \eqref{eq:SXrearr} by \(\tilde X^{(i)}(S_w^{(i)})^{-\top}\) yields
\begin{align*}
S_v^{(i)}
-\tilde X^{(i)}(S_w^{(i)})^{-\top}(L^{(i)})^\top L^{(i)}
&=
\tilde X^{(i)}(S_w^{(i)})^{-\top}(\tilde X^{(i)})^{-1}+
\tilde X^{(i)}(S_w^{(i)})^{-\top}\hat B_r^{(i)}R^{-1}C_r^{(i)}.
\end{align*}
Using \eqref{free_par} together with
\[
A_r^{(i)}=S_v^{(i)}-B_r^{(i)}L^{(i)},
\]
this becomes
\begin{equation}
A_r^{(i)}
=
\tilde X^{(i)}(S_w^{(i)})^{-\top}(\tilde X^{(i)})^{-1}
\bigl(I+\tilde X^{(i)}\hat B_r^{(i)}R^{-1}C_r^{(i)}\bigr).
\label{eq:Arfactor}
\end{equation}

Postmultiplying \eqref{eq:SXrearr} by \((S_v^{(i)})^{-1}\tilde X^{(i)}\) gives
\begin{align*}
(S_w^{(i)})^\top
-(L^{(i)})^\top L^{(i)}(S_v^{(i)})^{-1}\tilde X^{(i)}
&=
(\tilde X^{(i)})^{-1}(S_v^{(i)})^{-1}\tilde X^{(i)}+
\hat B_r^{(i)}R^{-1}C_r^{(i)}(S_v^{(i)})^{-1}\tilde X^{(i)}.
\end{align*}
Using \eqref{free_par} together with
\[
\hat A_r^{(i)}=(S_w^{(i)})^\top-(L^{(i)})^\top\hat C_r^{(i)},
\]
this becomes
\begin{equation}
\hat A_r^{(i)}
=
\bigl(I+\hat B_r^{(i)}R^{-1}C_r^{(i)}\tilde X^{(i)}\bigr)
(\tilde X^{(i)})^{-1}(S_v^{(i)})^{-1}\tilde X^{(i)}.
\label{eq:hatArfactor}
\end{equation}

By the Woodbury identity \cite{golub2013matrix}, and using the invertibility of \(R\) and
\(R+C_r^{(i)}\tilde X^{(i)}\hat B_r^{(i)}\),
\begin{align}
I-\tilde X^{(i)}\hat B_r^{(i)}
\bigl(R+C_r^{(i)}\tilde X^{(i)}\hat B_r^{(i)}\bigr)^{-1}
C_r^{(i)}
&=
\bigl(I+\tilde X^{(i)}\hat B_r^{(i)}R^{-1}C_r^{(i)}\bigr)^{-1},
\label{eq:Woodbury1}
\\
I-\hat B_r^{(i)}
\bigl(R+C_r^{(i)}\tilde X^{(i)}\hat B_r^{(i)}\bigr)^{-1}
C_r^{(i)}\tilde X^{(i)}
&=
\bigl(I+\hat B_r^{(i)}R^{-1}C_r^{(i)}\tilde X^{(i)}\bigr)^{-1}.
\label{eq:Woodbury2}
\end{align}
Therefore, using \eqref{eq:Arfactor} and \eqref{eq:Woodbury1},
\begin{equation}
\begin{aligned}
A_{r,cl}^{(i)}
&=
A_r^{(i)}
\Bigl(
I-\tilde X^{(i)}\hat B_r^{(i)}
\bigl(R+C_r^{(i)}\tilde X^{(i)}\hat B_r^{(i)}\bigr)^{-1}
C_r^{(i)}
\Bigr)
\\
&=
A_r^{(i)}
\bigl(I+\tilde X^{(i)}\hat B_r^{(i)}R^{-1}C_r^{(i)}\bigr)^{-1}
\\
&=
\tilde X^{(i)}(S_w^{(i)})^{-\top}(\tilde X^{(i)})^{-1}.
\end{aligned}
\label{eq:Aclsimilarity}
\end{equation}
Similarly, using \eqref{eq:hatArfactor} and \eqref{eq:Woodbury2},
\begin{equation}
\begin{aligned}
\hat A_{r,cl}^{(i)}
&=
\Bigl(
I-\hat B_r^{(i)}
\bigl(R+C_r^{(i)}\tilde X^{(i)}\hat B_r^{(i)}\bigr)^{-1}
C_r^{(i)}\tilde X^{(i)}
\Bigr)
\hat A_r^{(i)}
\\
&=
\bigl(I+\hat B_r^{(i)}R^{-1}C_r^{(i)}\tilde X^{(i)}\bigr)^{-1}
\hat A_r^{(i)}
\\
&=
(\tilde X^{(i)})^{-1}(S_v^{(i)})^{-1}\tilde X^{(i)}.
\end{aligned}
\label{eq:hatAclsimilarity}
\end{equation}
Thus, \(A_{r,cl}^{(i)}\) is similar to \((S_w^{(i)})^{-\top}\), and \(\hat A_{r,cl}^{(i)}\) is similar to \((S_v^{(i)})^{-1}\). Since the eigenvalues of \(S_w^{(i)}\) are \(-\beta_1,\dots,-\beta_i\) and the eigenvalues of \(S_v^{(i)}\) are \(-\alpha_1,\dots,-\alpha_i\), each with multiplicity \(m\), the eigenvalues of \(A_{r,cl}^{(i)}\) are
\[
-\frac{1}{\beta_1},\dots,-\frac{1}{\beta_i},
\]
and the eigenvalues of \(\hat A_{r,cl}^{(i)}\) are
\[
-\frac{1}{\alpha_1},\dots,-\frac{1}{\alpha_i},
\]
each with multiplicity \(m\). Because \(|\alpha_j|>1\) and \(|\beta_j|>1\), both closed-loop matrices are Schur stable.

It remains to verify that \(\tilde X^{(i)}\) satisfies the projected DTNARE \eqref{proj_dtnare}. Evaluating the left-hand side of \eqref{proj_dtnare} at \(\tilde X^{(i)}\) gives
\begin{align*}
&
A_r^{(i)}\tilde X^{(i)}\hat A_r^{(i)}
-\tilde X^{(i)}
-A_r^{(i)}\tilde X^{(i)}\hat B_r^{(i)}
\bigl(R+C_r^{(i)}\tilde X^{(i)}\hat B_r^{(i)}\bigr)^{-1}
C_r^{(i)}\tilde X^{(i)}\hat A_r^{(i)}
+B_r^{(i)}\hat C_r^{(i)}
\\
&=
\Bigl(
A_r^{(i)}
-A_r^{(i)}\tilde X^{(i)}\hat B_r^{(i)}
\bigl(R+C_r^{(i)}\tilde X^{(i)}\hat B_r^{(i)}\bigr)^{-1}
C_r^{(i)}
\Bigr)
\tilde X^{(i)}\hat A_r^{(i)}
-\tilde X^{(i)}
+B_r^{(i)}\hat C_r^{(i)}
\\
&=
A_{r,cl}^{(i)}\tilde X^{(i)}\hat A_r^{(i)}
-\tilde X^{(i)}
+B_r^{(i)}\hat C_r^{(i)}
\\
&=
\tilde X^{(i)}(S_w^{(i)})^{-\top}(\tilde X^{(i)})^{-1}
\tilde X^{(i)}\hat A_r^{(i)}
-\tilde X^{(i)}
+B_r^{(i)}\hat C_r^{(i)}
\\
&=
\tilde X^{(i)}(S_w^{(i)})^{-\top}\hat A_r^{(i)}
-\tilde X^{(i)}
+B_r^{(i)}\hat C_r^{(i)}
\\
&=
\tilde X^{(i)}(S_w^{(i)})^{-\top}
\bigl((S_w^{(i)})^\top-(L^{(i)})^\top\hat C_r^{(i)}\bigr)
-\tilde X^{(i)}
+B_r^{(i)}\hat C_r^{(i)}
\\
&=
\tilde X^{(i)}
-\tilde X^{(i)}(S_w^{(i)})^{-\top}(L^{(i)})^\top\hat C_r^{(i)}
-\tilde X^{(i)}
+B_r^{(i)}\hat C_r^{(i)}
\\
&=0,
\end{align*}
where the final equality follows from
\[
B_r^{(i)}
=
\tilde X^{(i)}(S_w^{(i)})^{-\top}(L^{(i)})^\top.
\]
Thus, \(\tilde X^{(i)}\) is a solution of the projected DTNARE \eqref{proj_dtnare}. Since the corresponding closed-loop matrices are Schur stable, \(\tilde X^{(i)}\) is a stabilizing solution.

For the residual formula, first observe that the recursion for \(M^{(i)}\) follows directly from the block structures
\[
\tilde X^{(i)}=
\begin{bmatrix}
\tilde X^{(i-1)} & 0\\
0 & \tilde x_i
\end{bmatrix},
\qquad
B_r^{(i)}=
\begin{bmatrix}
B_r^{(i-1)}\\
b_i
\end{bmatrix},
\qquad
\hat C_r^{(i)}=
\begin{bmatrix}
\hat C_r^{(i-1)} & \hat c_i
\end{bmatrix}.
\]
Indeed,
\begin{align*}
\hat C_r^{(i)}(\tilde X^{(i)})^{-1}B_r^{(i)}
&=
\hat C_r^{(i-1)}(\tilde X^{(i-1)})^{-1}B_r^{(i-1)}
+
\hat c_i\tilde x_i^{-1}b_i.
\end{align*}
Hence,
\[
M^{(i)}
=
I_m-\hat C_r^{(i)}(\tilde X^{(i)})^{-1}B_r^{(i)}
=
M^{(i-1)}-\hat c_i\tilde x_i^{-1}b_i,
\qquad
M^{(0)}=I_m.
\]

Next, note that
\begin{align*}
B\hat{C}=\big(B_\perp^{(i)}+EV^{(i)}B_r^{(i)}\big)
\big(\hat C_\perp^{(i)}+\hat C_r^{(i)}(\hat W^{(i)})^\top\hat E\big).
\end{align*}
Substituting this into \eqref{residual} gives
\begin{align*}
R_x^{(i)}
&=
EV^{(i)}
\Bigl[
A_r^{(i)}\tilde X^{(i)}\hat A_r^{(i)}
-\tilde X^{(i)}
-A_r^{(i)}\tilde X^{(i)}\hat B_r^{(i)}
\bigl(R+C_r^{(i)}\tilde X^{(i)}\hat B_r^{(i)}\bigr)^{-1}
C_r^{(i)}\tilde X^{(i)}\hat A_r^{(i)}
+B_r^{(i)}\hat C_r^{(i)}
\Bigr]
(\hat W^{(i)})^\top\hat E
\\
&\quad+
EV^{(i)}
\Bigl[
-A_r^{(i)}\tilde X^{(i)}(L^{(i)})^\top
+A_r^{(i)}\tilde X^{(i)}\hat B_r^{(i)}
\bigl(R+C_r^{(i)}\tilde X^{(i)}\hat B_r^{(i)}\bigr)^{-1}
C_r^{(i)}\tilde X^{(i)}(L^{(i)})^\top
+B_r^{(i)}
\Bigr]
\hat C_\perp^{(i)}
\\
&\quad+
B_\perp^{(i)}
\Bigl[
-L^{(i)}\tilde X^{(i)}\hat A_r^{(i)}
+L^{(i)}\tilde X^{(i)}\hat B_r^{(i)}
\bigl(R+C_r^{(i)}\tilde X^{(i)}\hat B_r^{(i)}\bigr)^{-1}
C_r^{(i)}\tilde X^{(i)}\hat A_r^{(i)}
+\hat C_r^{(i)}
\Bigr]
(\hat W^{(i)})^\top\hat E
\\
&\quad+
B_\perp^{(i)}
\Bigl[
I_m
+L^{(i)}\tilde X^{(i)}(L^{(i)})^\top
-L^{(i)}\tilde X^{(i)}\hat B_r^{(i)}
\bigl(R+C_r^{(i)}\tilde X^{(i)}\hat B_r^{(i)}\bigr)^{-1}
C_r^{(i)}\tilde X^{(i)}(L^{(i)})^\top
\Bigr]
\hat C_\perp^{(i)}.
\end{align*}

The first bracket vanishes because \(\tilde X^{(i)}\) satisfies the projected DTNARE \eqref{proj_dtnare}.

For the second bracket, apply the Woodbury identity \cite{golub2013matrix} in the form
\[
I-\hat B_r^{(i)}
\bigl(R+C_r^{(i)}\tilde X^{(i)}\hat B_r^{(i)}\bigr)^{-1}
C_r^{(i)}\tilde X^{(i)}
=
\bigl(I+\hat B_r^{(i)}R^{-1}C_r^{(i)}\tilde X^{(i)}\bigr)^{-1}.
\]
Then the second bracket becomes
\[
-A_r^{(i)}\tilde X^{(i)}
\bigl(I+\hat B_r^{(i)}R^{-1}C_r^{(i)}\tilde X^{(i)}\bigr)^{-1}
(L^{(i)})^\top
+B_r^{(i)}.
\]
Using the factorization
\[
A_r^{(i)}
=
\tilde X^{(i)}(S_w^{(i)})^{-\top}(\tilde X^{(i)})^{-1}
\bigl(I+\tilde X^{(i)}\hat B_r^{(i)}R^{-1}C_r^{(i)}\bigr),
\]
together with
\[
\tilde X^{(i)}
\bigl(I+\hat B_r^{(i)}R^{-1}C_r^{(i)}\tilde X^{(i)}\bigr)^{-1}
=
\bigl(I+\tilde X^{(i)}\hat B_r^{(i)}R^{-1}C_r^{(i)}\bigr)^{-1}
\tilde X^{(i)},
\]
one obtains
\[
A_r^{(i)}\tilde X^{(i)}
\bigl(I+\hat B_r^{(i)}R^{-1}C_r^{(i)}\tilde X^{(i)}\bigr)^{-1}
=
\tilde X^{(i)}(S_w^{(i)})^{-\top}.
\]
Thus, the second bracket is
\[
-\tilde X^{(i)}(S_w^{(i)})^{-\top}(L^{(i)})^\top
+B_r^{(i)}
=0,
\]
because
\[
B_r^{(i)}=\tilde X^{(i)}(S_w^{(i)})^{-\top}(L^{(i)})^\top.
\]

For the third bracket, the same Woodbury identity gives
\[
-L^{(i)}\tilde X^{(i)}
\bigl(I+\hat B_r^{(i)}R^{-1}C_r^{(i)}\tilde X^{(i)}\bigr)^{-1}
\hat A_r^{(i)}
+\hat C_r^{(i)}.
\]
Using the factorization
\[
\hat A_r^{(i)}
=
\bigl(I+\hat B_r^{(i)}R^{-1}C_r^{(i)}\tilde X^{(i)}\bigr)
(\tilde X^{(i)})^{-1}(S_v^{(i)})^{-1}\tilde X^{(i)},
\]
it follows that
\[
\bigl(I+\hat B_r^{(i)}R^{-1}C_r^{(i)}\tilde X^{(i)}\bigr)^{-1}
\hat A_r^{(i)}
=
(\tilde X^{(i)})^{-1}(S_v^{(i)})^{-1}\tilde X^{(i)}.
\]
Therefore, the third bracket becomes
\[
-L^{(i)}\tilde X^{(i)}
(\tilde X^{(i)})^{-1}(S_v^{(i)})^{-1}\tilde X^{(i)}
+\hat C_r^{(i)}
=
-L^{(i)}(S_v^{(i)})^{-1}\tilde X^{(i)}
+\hat C_r^{(i)}
=0,
\]
because
\[
\hat C_r^{(i)}=L^{(i)}(S_v^{(i)})^{-1}\tilde X^{(i)}.
\]

It remains to simplify the fourth bracket. By the Woodbury identity,
\[
I_m
+L^{(i)}\tilde X^{(i)}(L^{(i)})^\top
-L^{(i)}\tilde X^{(i)}\hat B_r^{(i)}
\bigl(R+C_r^{(i)}\tilde X^{(i)}\hat B_r^{(i)}\bigr)^{-1}
C_r^{(i)}\tilde X^{(i)}(L^{(i)})^\top
\]
is equal to
\[
I_m+
L^{(i)}\tilde X^{(i)}
\bigl(I+\hat B_r^{(i)}R^{-1}C_r^{(i)}\tilde X^{(i)}\bigr)^{-1}
(L^{(i)})^\top.
\]
This matrix is \((M^{(i)})^{-1}\). Using
\[
M^{(i)}
=
I_m-\hat C_r^{(i)}(\tilde X^{(i)})^{-1}B_r^{(i)}
=
I_m-
L^{(i)}(S_v^{(i)})^{-1}\tilde X^{(i)}(S_w^{(i)})^{-\top}(L^{(i)})^\top,
\]
a direct multiplication gives
\begin{align*}
&\bigl(I_m-\hat C_r^{(i)}(\tilde X^{(i)})^{-1}B_r^{(i)}\bigr)
\Bigl[
I_m+
L^{(i)}\tilde X^{(i)}
\bigl(I+\hat B_r^{(i)}R^{-1}C_r^{(i)}\tilde X^{(i)}\bigr)^{-1}
(L^{(i)})^\top
\Bigr]
\\
&=
I_m
+L^{(i)}\tilde X^{(i)}
\bigl(I+\hat B_r^{(i)}R^{-1}C_r^{(i)}\tilde X^{(i)}\bigr)^{-1}
(L^{(i)})^\top
\\
&\quad-
L^{(i)}(S_v^{(i)})^{-1}\tilde X^{(i)}(S_w^{(i)})^{-\top}(L^{(i)})^\top
\\
&\quad-
L^{(i)}(S_v^{(i)})^{-1}\tilde X^{(i)}(S_w^{(i)})^{-\top}
(L^{(i)})^\top
L^{(i)}\tilde X^{(i)}
\bigl(I+\hat B_r^{(i)}R^{-1}C_r^{(i)}\tilde X^{(i)}\bigr)^{-1}
(L^{(i)})^\top
\\
&=
I_m
+L^{(i)}
\Bigl[
\tilde X^{(i)}
\bigl(I+\hat B_r^{(i)}R^{-1}C_r^{(i)}\tilde X^{(i)}\bigr)^{-1}
-(S_v^{(i)})^{-1}\tilde X^{(i)}(S_w^{(i)})^{-\top}
\\
&\qquad\qquad-
(S_v^{(i)})^{-1}\tilde X^{(i)}(S_w^{(i)})^{-\top}
(L^{(i)})^\top L^{(i)}\tilde X^{(i)}
\bigl(I+\hat B_r^{(i)}R^{-1}C_r^{(i)}\tilde X^{(i)}\bigr)^{-1}
\Bigr]
(L^{(i)})^\top.
\end{align*}
The expression inside the brackets is zero. Indeed, since
\[
A_r^{(i)}
=
S_v^{(i)}-B_r^{(i)}L^{(i)}
=
S_v^{(i)}-
\tilde X^{(i)}(S_w^{(i)})^{-\top}(L^{(i)})^\top L^{(i)}
\]
and
\[
A_r^{(i)}\tilde X^{(i)}
\bigl(I+\hat B_r^{(i)}R^{-1}C_r^{(i)}\tilde X^{(i)}\bigr)^{-1}
=
\tilde X^{(i)}(S_w^{(i)})^{-\top},
\]
one has
\[
\Bigl(
S_v^{(i)}-
\tilde X^{(i)}(S_w^{(i)})^{-\top}(L^{(i)})^\top L^{(i)}
\Bigr)
\tilde X^{(i)}
\bigl(I+\hat B_r^{(i)}R^{-1}C_r^{(i)}\tilde X^{(i)}\bigr)^{-1}
=
\tilde X^{(i)}(S_w^{(i)})^{-\top}.
\]
Premultiplying by \((S_v^{(i)})^{-1}\) yields
\[
\tilde X^{(i)}
\bigl(I+\hat B_r^{(i)}R^{-1}C_r^{(i)}\tilde X^{(i)}\bigr)^{-1}
-(S_v^{(i)})^{-1}\tilde X^{(i)}(S_w^{(i)})^{-\top}
\]
\[
-
(S_v^{(i)})^{-1}\tilde X^{(i)}(S_w^{(i)})^{-\top}
(L^{(i)})^\top L^{(i)}\tilde X^{(i)}
\bigl(I+\hat B_r^{(i)}R^{-1}C_r^{(i)}\tilde X^{(i)}\bigr)^{-1}
=0.
\]
Hence,
\[
\bigl(I_m-\hat C_r^{(i)}(\tilde X^{(i)})^{-1}B_r^{(i)}\bigr)
\Bigl[
I_m+
L^{(i)}\tilde X^{(i)}
\bigl(I+\hat B_r^{(i)}R^{-1}C_r^{(i)}\tilde X^{(i)}\bigr)^{-1}
(L^{(i)})^\top
\Bigr]
=I_m.
\]
Since the matrices are square, the fourth bracket is the inverse of
\[
M^{(i)}=I_m-\hat C_r^{(i)}(\tilde X^{(i)})^{-1}B_r^{(i)}.
\]

Consequently,
\[
R_x^{(i)}
=
B_\perp^{(i)}
\bigl(M^{(i)}\bigr)^{-1}
\hat C_\perp^{(i)}.
\]

Finally, the Petrov--Galerkin condition holds because
\[
\hat C_\perp^{(i)}\hat V^{(i)}
=
\hat C\hat V^{(i)}
-\hat C_r^{(i)}(\hat W^{(i)})^\top\hat E\hat V^{(i)}
=
\hat C_r^{(i)}-\hat C_r^{(i)}
=0,
\]
and therefore
\[
(W^{(i)})^\top R_x^{(i)}\hat V^{(i)}
=
(W^{(i)})^\top
B_\perp^{(i)}
\bigl(M^{(i)}\bigr)^{-1}
\hat C_\perp^{(i)}\hat V^{(i)}
=0.
\]

For the last statement, consider
\begin{align*}
& A V^{(i)} - E V^{(i)} A_r^{(i)} + B_{\perp}^{(i)} L^{(i)} \\
&= A V^{(i)} - E V^{(i)} \big(S_v^{(i)} - B_r^{(i)} L^{(i)}\big) + \big(B - E V^{(i)} B_r^{(i)}\big) L^{(i)} \\
&= 0.
\end{align*}
This establishes \eqref{Sylv_V2}. By duality, the same argument shows that \(\hat{W}^{(i)}\) satisfies \eqref{Sylv_W2}.

This completes the proof.
\end{proof}

\end{document}